\documentclass{article}

\usepackage[english]{babel}
\usepackage[utf8x]{inputenc}
\usepackage[T1]{fontenc}
\usepackage{graphicx}
\usepackage{hyperref}
\usepackage{caption}
\usepackage{psfrag}

\usepackage{amsthm}
\usepackage{amssymb}
\usepackage{amsmath}
\usepackage{amsfonts}

\newtheorem{thm}{Theorem}[section]
\newtheorem{prop}[thm]{Proposition}
\newtheorem{rem}[thm]{Remark}
\newtheorem{ex}[thm]{Example}
\newcommand{\argmin}{\text{argmin}}

\newcommand{\N}{\mathbb{N}}
\newcommand{\R}{\mathbb{R}}
\newcommand{\eps}{\epsilon}
\newcommand{\sign}{\text{sign}}
\newcommand{\de}{\mathrel{\mathop:}\hspace*{-.6pt}=}

\begin{document}
	
\title{Numerical Approximation of Hyperbolic Systems Containing an Interface}
\author{Nina Aguillon \footnotemark[1]\and  Raul Borsche  \footnotemark[2]}
\footnotetext[1]{Centre de Math\'ematiques et Informatique, 39 rue Joliot-Curie, 13453 Marseille Cedex 13, France 
	}
\footnotetext[2]{Technische Universit\"at Kaiserslautern, Department of Mathematics, Erwin-Schr\"odinger-Stra{\ss}e, 67663 Kaiserslautern, Germany 
	borsche@mathematik.uni-kl.de}

\maketitle
\begin{abstract}
	In this paper we present an approach to approximate numerically the solution of coupled hyperbolic conservation laws. 
	The coupling is achieved through a fixed interface, in which interface conditions are linking the traces of both sides.
	The numerical solver is based on central methods, like the Rusanov scheme, and does not use the structure of the Riemann Problem. 
	It consists in balancing the effects of the waves that enter the interface. 
	The scheme is well balanced with respect to all the piecewise constant equilibria associated with the interface condition and is able to maintain exactly conservation properties of the interface conditions.
	A detailed analysis and several numerical tests show the quality of the method. 
	Different applications, including sonic and transsonic flows and a multiphysic model are studied. 
\end{abstract}

In this paper we propose a simple numerical scheme to approximate the solution of the Cauchy problem

\begin{equation} \label{eq:HypSysInter} 
\begin{cases}
\partial_{t} U(t,x) + \partial_{x} f(U(t,x)) = 0, & \text{ for } x \in \R \setminus \{0\}, \\
\big( U(t,0_{-}),U(t,0_{+}) \big) \in \mathcal{G}, & \text{ a.e. } t>0, \\
U(0,x)= U^{0}(x).
\end{cases}
\end{equation}
Special attention is drawn to the point $x=0$, where the interface conditions $\mathcal{G}$ link the traces of the states on the left and right hand side.

This system can be used to describe e.g. the flow in pipes or channels with varying cross section \cite{ColomboMarcellini,GarciaNavarroHubbard}.
It can also be viewed as a particular coupling in networks of conservation laws \cite{MR2223073,CHS08,MR3195345}.
In the case where the equation~\eqref{eq:HypSysInter} is scalar, the theories of $L^{1}$-dissipative germ~\cite{AKR11} and of transmission maps~\cite{AC15} explore the links between the shape of $\mathcal{G}$ and the well-posedness of~(\ref{eq:HypSysInter}-\ref{eq:EI}).

Numerically speaking, a large class of efficient numerical schemes is available to solve~\eqref{eq:HypSysInter} away from the interface. 
However, the development of general numerical methods for the junction at $x=0$ remains an open challenge. 
Most of the present solvers rely on solving Riemann problems~\eqref{eq:RP} associated to~\eqref{eq:HypSysInter}, i.e. constant initial data on each side of the interface \cite{MR3195345,MR3328156,BKlar14,MR2223073,MR3237558, GLTR05}.
Extensions of such schemes to higher order have been studied in \cite{BK14,BK15}. Recently, relaxation procedures adapted to handle coupling conditions have been proposed for different models, see~\cite{CGHMR15, CKS14, ACCG14, ACCFGLRS07, ACCGLRS08}.

Solving such Riemann problems exactly requires a detailed knowledge of the structure of the solution to~\eqref{eq:HypSysInter}.
For complicated coupling scenarios these informations might be not accessible or only computable with high numerical costs.
Thus it seems essential to seek for numerical methods that do not require any information on the structure of the Riemann problem for~\eqref{eq:HypSysInter}. 
In that direction, the scalar case is particularly well studied.
Examples of such couplings are the fluid-particle toy-model of~\cite{AS12} or the general setting of scalar conservation laws with discontinuous flux, see~\cite{AKR11}.
 
Recently, in~\cite{Bor15} a systematic way to approximate the solution near the interface has been proposed.
The idea is to mimic numerically as close as possible the structure of the underlying Riemann problem.
This includes a numerical minimization of all waves which do not occur in the exact solution of the Riemann problem.

The approach in the present paper is to allow unphysical waves inside the node, but we require that their contributions cancel in the node.
Due to this modification the scheme simplifies significantly and the exact conservation property of the exact solution is restored.

The paper is organized as follows. 
In the first section, we present several sets of interface conditions for models of fluids dynamics, which illustrates the variety of situation described by system~\eqref{eq:HypSysInter}. 
In the second section, we give the mathematical framework in which~\eqref{eq:HypSysInter} is solved. 
We recall some results concerning the Riemann problem and explain the derivation of the new scheme. 
In the third section, we investigate the numerical equilibrium states associated to the proposed method, both for the Godunov and the Rusanov fluxes.
Section~\ref{S:Rusanov} contains a complete study of the scheme in the context of the Rusanov flux for a $2 \times 2$ model of fluid / particle interaction introduced in~\cite{Agu15}.
In the last Section, we apply this method to several examples, including  a new fluid / particle problem with heat exchange. 
This illustrates the interest of using a Riemann-problem free method, as no details of the Riemann problem for this model are available at the moment.
Furthermore test cases with different pressure laws or varying cross section are considered.

\section{Interface conditions}
In the sequel of this paper we will consider five representatives of the coupling conditions $\mathcal{G}$.

Let us first consider the Cauchy problem 
\begin{equation} \label{eq:HypSys} 
\begin{cases}
\partial_{t} U(t,x) + \partial_{x} f(U(t,x)) = 0 & \text{ for } x \in \R, \\
U(0,x)= U^{0}(x).
\end{cases}
\end{equation}
The solution $U$ takes its values in an open convex subset $\Omega$ of $\R^{n}$ and the flux $f: \Omega \subset \R^{n} \rightarrow \R^{n}$ is regular. 
The system is assumed to be strictly hyperbolic, i.e. the Jacobian matrix of $f$ is diagonalizable with $n$ distinct eigenvalues. 
Recall that the solution is not unique in the class of weak functions, but the well-posedness for~\eqref{eq:HypSys} is typically recovered by additionally enforcing an entropy inequality
\begin{equation} \label{eq:EI}
\partial_{t} E(U(t,x)) + \partial_{x}(F(U(t,x))) \leq 0, 
\end{equation} 
where the entropy $E$ is supposed to be regular and strictly convex. 
Therefore in the sequel we consider solutions of~\eqref{eq:HypSysInter} that verify~\eqref{eq:EI} on the sets $\{t>0, x<0\}$ and $\{t>0, x>0\}$. 
A natural extension of such a constraint to the coupling point is given by  
$$ F(U_{-})-F(U_{+}) \leq 0\ , $$
which guarantees that the entropy is non-increasing in the coupled system.

\subsection{Classical coupling}
The first choice models the classical solutions on a continuous line, which enters the larger framework of~\eqref{eq:HypSysInter} as
\begin{equation} \label{eq:Gcla}
\mathcal{G}_0= \{(U_{-},U_{+}) \in \Omega^{2}, \ f(U_{-})=f(U_{+}) \ \text{ and } \ F(U_{+})-F(U_{-}) \leq 0 \}.
\end{equation}
The solution of~\eqref{eq:HypSysInter} with the above coupling coincides with the solution of~\eqref{eq:HypSys}. 
Thus, this example provides a natural way to compare the coupling procedure to standard schemes for classical Riemann problems.

\subsection{Fluid/particle coupling}
The second set of coupling conditions arises from the modeling of a particle in a surrounding fluid~\cite{Agu15}.
The motion of the fluid is described by the isothermal Euler equations
$$ U=(\rho,q), \quad \Omega=\R^{+}_{\star} \times \R, \quad f(U)=\left(q, \frac{q^{2}}{\rho} + c^{2} \rho \right), $$
where $\R^{+}_{\star}=(0, +\infty)$. 
The corresponding entropy/entropy-flux pair is given by
$ E(U)= \frac{q^{2}}{2 \rho} + c^{2} \rho \log(\rho)$ and $F(U)=\frac{q}{\rho} \left( E(\rho,q) +c^{2} \rho \right)$.

At $x=0$ a fixed obstacle is located in the fluid. 
The fluid can pass with a certain resistance through the obstacle. 
As no fluid disappears, the total mass is conserved across the interface,
whereas the momentum decreases.
This can be described by the following coupling conditions
\begin{equation} \label{eq: Gpart}
  \mathcal{G}_{\lambda} = 
\begin{pmatrix} 
q_{-}=q_{+}=: q \\
\left( \frac{q^{2}}{\rho_{-}}+c^{2} \rho_{-} \right) -  \left( \frac{q^{2}}{\rho_{+}}+c^{2} \rho_{+} \right) = \lambda q \\
0 \leq q \leq c \rho_{-} \Longrightarrow  0 \leq q \leq c \rho_{+} \\ 
-c \rho_{+} \leq q \leq 0 \Longrightarrow  -c \rho_{-} \leq q \leq 0
\end{pmatrix},
\end{equation}
where $\lambda$ is a positive friction parameter representing the resistance of the obstacle.
The last two conditions can be understood as entropy conditions. 
The Riemann problem for this interface conditions have been studied intensively in~\cite{Agu15}.
Its solution and a variety of test cases are available, as well as numerical methods based on the solution of the Riemann problem, see~\cite{Agu14}. 
A simplification of this model is described in~\cite{LST08}.

\subsection{Fluid/particle coupling with heat exchange}\label{S:HeatExchange}

We now propose an extension of the above model to the case where the fluid has a varying temperature and can exchange heat with the obstacle.
The motion of the fluid is described by the Euler equations
$$ 
U=(\rho,q,E),
 \quad 
 \Omega=\R^{+}_{\star} \times \R \times \R^{+}_{\star},
 \quad 
 f(U)=\left(q, \frac{q^{2}}{\rho} + p,\frac{q}{\rho}(E+p) \right). 
$$
The system is closed with the ideal gas law
$p= e \rho (\gamma -1)$ with the adiabatic exponent $\gamma>1$.
Here $\rho$ denotes the density, $u$ velocity, $p$ pressure, $e$ internal energy and 
$E=\frac{1}{2} \rho u^{2} + \rho e$ the total energy.

For smooth solutions we can describe the influence of the particle on the fluid as source terms. 
These can be derived by considering the balance of forces and the total energy in the coupled system
\begin{equation} \label{eq:heatex}
\begin{cases}
\partial_{t} \rho + \partial_{x} (\rho u)= 0, \\
\partial_{t}(\rho u) + \partial_{x}( \rho u^{2} + p) = -\lambda \rho u \delta_{0}, \\
\partial_{t} E + \partial_{x}(u(E+p)) = - \lambda \rho u^{2} \delta_{0} - \mu \left( e- s_{P}\left( \frac{\rho}{\rho_{0}}\right)^{\gamma-1} \right)  \delta_{0}.
\end{cases}
\end{equation}
The force acting on the particle, $D=\lambda \rho u$, is proportional to the friction parameter $\lambda \geq 0$. It is located only at $x=0$, which is described by the Dirac measure $\delta_{0}(x)$. 

The work applied by this force,  $\lambda \rho u^{2}$, also appears in the energy balance. 
The term $ \mu \left( e- s_{P} \left( \frac{\rho}{\rho_{0}}\right)^{\gamma-1} \right)$ represents the heat exchange with the particle. It is described by a heat exchange parameter $\mu \geq 0$, a reference density $\rho_{0}$ and $s_{P} \geq 0$ has the dimension of an entropy.

Since the solutions of the Euler equations can not be expected to be continuous, the definition of the above source terms is complicated. 
Therefore we study a regularization of~\eqref{eq:heatex} where the Dirac measure is replaced by one of its regularization. 
It appears that the properties of a continuous and stationary fluid on the left and on the right of the particle are linked by universal relations, independent of the regularization of the Dirac mass.
\begin{prop} \label{P:CondInter}
	Let $\eps$ be a positive real and $H_{\eps}$ be a $\mathcal{C}^{1}$-regular function, increasing monotone from $0$ to $1$ on the interval $[-\eps, \eps]$. 
	Consider $x \mapsto (\rho_{\eps}(x), u_{\eps}(x), \allowbreak p_{\eps}(x))$  a $\mathcal{C}^{1}$-regular stationary solution of~\eqref{eq:heatex} where the Dirac mass is replaced by $H_{\eps}'$
	\begin{equation} \label{eq:regul}
	\begin{cases}
	(\rho_{\eps} u_{\eps})'= 0, \\
	(\rho_{\eps} u_{\eps}^{2}+ p_{\eps})'=-\lambda \rho_{\eps} u_{\eps} H_{\eps}', \\
	(u_{\eps}(E_{\eps}+p_{\eps}))' = - \lambda \rho_{\eps} u_{\eps}^{2} H_{\eps}' - \mu\left(e_{\eps}-s_{P} \left( \frac{\rho_{\eps}}{\rho_{0}}\right)^{\gamma-1} \right) H_{\eps}'.
	\end{cases}
	\end{equation}
	Then the states
	$ U_-= U_{\eps}(-\eps)$ and $U_{+}= U_{\eps}(\eps)$
	verify the following relations, independent of $\eps$ and $H_{\eps}$:
	$$\rho_{-}u_{-}=\rho_{+} u_{+} \de q, $$
	$$ \left( \frac{q^{2}}{\rho_{-}}+ p_{-}\right) - \left( \frac{q^{2}}{\rho_{+}}+ p_{+}\right)= \lambda q, $$
	and
	$$ (s_{+}-s_{P})= \exp \left(- \frac{\mu}{q} \right)(s_{-}-s_{P}), \ \text{ where }  s= e \rho^{1-\gamma}. $$
\end{prop}
\begin{proof}
	The first line of~\eqref{eq:regul} exactly says that the momentum $q_{\eps}=\rho_{\eps} u_{\eps}$ is constant across the particle. 
	The second line of~\eqref{eq:regul} reads
	\begin{equation} \label{eq:charge}
	\left( \frac{q_{\eps}^{2}}{\rho_{\eps}}+p_{\eps} \right)' = -\lambda q_{\eps} H_{\eps}',
	\end{equation}
	which yields the second relation by integrating on $[-\eps, \eps]$.
	
	Let us now focus on the last line of~\eqref{eq:regul}. 
	Replacing $E_{\eps}$ by $\frac{1}{2}q_{\eps} u_{\eps} + \rho_{\eps} e_{\eps}$ and using the fact that $q_{\eps}$ is constant, we obtain
	$$ 
	\begin{aligned}
	u_{\eps}' & \left( \frac{1}{2} q_{\eps} u_{\eps} + \rho_{\eps} e_{\eps} + p_{\eps} \right) + u_{\eps} \left( \frac{1}{2} q_{\eps} u_{\eps}' + \rho_{\eps}' e_{\eps} + \rho_{\eps} e_{\eps}'+ p_{\eps}' \right) \\
	&= -\lambda q_{\eps} u_{\eps} H_{\eps}' - \mu \left(e_{\eps}-s_{P} \left( \frac{\rho_{\eps}}{\rho_{0}}\right)^{\gamma-1} \right)H_{\eps}' .
	\end{aligned}
	$$
	Inserting the first two relations $(q_{\eps}u_{\eps}'+p_{\eps}')=-\lambda u_{\eps} H_{\eps}'$  and $q_{\eps}'=0$ into $ \rho_{\eps} u_{\eps}' +  \rho_{\eps}' u_{\eps} = 0$ we obtain
	\begin{equation} \label{eq:step}
	q_{\eps}e_{\eps}'+u_{\eps}'p_{\eps} = q_{\eps}e_{\eps}'- \frac{q_{\eps}}{\rho_{\eps}} \rho_{\eps}' (\gamma-1) e_{\eps}=  - \mu \left(e_{\eps}-s_{P} \left( \frac{\rho_{\eps}}{\rho_{0}}\right)^{\gamma-1} \right) H_{\eps}' \ .
	\end{equation}
	Introducing the quantity $s_{\eps}= e_{\eps} \left( \frac{\rho_{\eps}}{\rho_{0}} \right)^{1-\gamma}$ and computing its derivative gives
	$$s_{\eps}'= e_{\eps}' \left( \frac{\rho_{\eps}}{\rho_{0}} \right)^{1-\gamma} +(1-\gamma) e_{\eps} \frac{\rho_{\eps}^{-\gamma}}{\rho_{0}^{1-\gamma}} \rho_{\eps}' = \frac{\rho_{\eps}^{1-\gamma}}{\rho_{0}^{1-\gamma}} \left(e_{\eps}'-(\gamma-1)\frac{\rho_{\eps}'}{\rho_{\eps}}e_{\eps} \right). $$
	Thus~\eqref{eq:step} can be written as
	$$ q_{\eps} \left( \frac{ \rho_{\eps}}{\rho_{0}} \right)^{\gamma-1}  s_{\eps}'=  - \mu \left(e_{\eps}-s_{P} \left( \frac{\rho_{\eps}}{\rho_{0}}\right)^{\gamma-1} \right) H_{\eps}',  $$
	and simplified to
	$$ q_{\eps}(s_{\eps}-s_{P})'= -\mu (s_{\eps}-s_{P} ) H_{\eps}'.  $$
	If $q_{\eps}=0$, the only solution is $s=s_{P}$. 
	Moreover due to the Cauchy--Lipschitz theorem, $s_{\eps}-s_{P}$ does not change sign, and if it is not null we have
	$$ \log(|s_{\eps}-s_{P}|)'= -\frac{\mu}{q_{\eps}} H_{\eps}', $$
	and we can conclude the desired statement by integration.
\end{proof}

\begin{rem}
	In Proposition~\ref{P:CondInter}, the stationary solution is assumed to be $\mathcal{C}^{1}$-regular. 
	In particular, it excludes the case where a stationary entropy satisfying shock lies somewhere inside the thickened particle. 
	This can occur only when the flow is supersonic at the entrance of the particle. 
	Thus the interface conditions derived in Proposition~\ref{P:CondInter} are valid for subsonic flows only. 
	Allowing a stationary shock inside the particle (or in other words, looking for a \emph{piecewise} $\mathcal{C}^{1}$-regular stationary solution of~\eqref{eq:heatex}) yields to more complicated computations, because of a lack of compatibility between the regular part of the solution and the shocks. 
	In the case of an isothermal flow, a complete study has been done in~\cite{Agu14}.
\end{rem}

Following the above proposition we consider for the particle with heat exchange the coupling conditions 
\begin{equation} \label{eq:heat_CouplingConditions}
\begin{cases}
q_-=q_+, \\
\left( \frac{q^{2}}{\rho_{+}}+p_{+} \right) -  \left( \frac{q^{2}}{\rho_{-}}+p_{-} \right) = -\lambda q, \\
(s_{+}-s_{P})= \exp \left( \frac{\mu}{q} \right)(s_{-}-s_{P}),
\end{cases}
\end{equation}
with  $s= e \left(\frac{\rho}{\rho_{0}} \right)^{1-\gamma}$.
Note that balancing the jump in the flux with the right hand side of \eqref{eq:heatex} yields the problem of defining $\rho$ and $e$ inside of the jump.
This is avoided by using \eqref{eq:heat_CouplingConditions}.

\subsection{Gas dynamics with different pressure laws}
We also consider a case where the conservation law is different on each side of the interface
\begin{equation} \label{eq:diffpressure}
 \begin{cases}
 \partial_{t} U + \partial_{x} f_{L}(U)= 0 & \text{ on } x<0, \\
 \partial_{t} U + \partial_{x} f_{R}(U)= 0 & \text{ on } x>0,
\end{cases}
\end{equation}
where the flux function $f_{L}$ and $f_{R}$ are different, but the system is strictly hyperbolic for those two fluxes. We chose the model presented in~\cite{CGHMR15}, where the equation of gas dynamics is used on both side of the interface
$$ 
U=(\rho,q,E),
 \quad 
 \Omega=\R^{+}_{\star} \times \R \times \R^{+}_{\star},
 \quad 
 f_{L/R}(U)=\left(q, \frac{q^{2}}{\rho} + p,\frac{q}{\rho}(E+p_{L/R}) \right),
$$
but with a discontinuous pressure law $p_{L/R}= e \rho (\gamma_{L/R} -1)$ on $x<0$ and $x>0$. Following~\cite{CGHMR15}, we consider two different interface conditions. The first one is associated to the so-called ``flux coupling''
\begin{equation} \label{eq:fluxcoupling}
 \mathcal{G}_{\text{flux}}= \{(U_{-}, U_{+}) \in \Omega^{2}: f_{L}(U_{-})=f_{R}(U_{+})\},
\end{equation}
which yields the conservation of the density $\rho$, the momentum $q$ and the total energy $E$. The second one is  the so-called ``state coupling'' 
\begin{equation} \label{eq:statecoupling}
  \mathcal{G}_{\text{state}}= \{(U_{-}, U_{+}) \in \Omega^{2}: (\rho_{-}, u_{-}, p_{L,-}) =(\rho_{+}, u_{+}, p_{R,+})\}, 
\end{equation}
which ensures the continuity of the density $\rho$, the velocity $u$ and the pressure $p$. For the latter, the subscript $L$ and $R$ recall that the pressure law is different on each side of the interface, i.e the last equation of the coupling conditions reads
$$e_{-} \rho_{-} (\gamma_{L} -1) = e_{+} \rho_{+} (\gamma_{R} -1), $$
where $e= \frac{E}{\rho} - \frac{1}{2} u^{2}$ is the internal energy.

\subsection{Barotropic flows in a nozzle with piecewise constant cross-section} 
As last example we consider the following model~\cite{CKS14}:
\begin{equation} \label{eq:nozzle}
\begin{cases}
 \partial_{t} \alpha \rho + \partial_{x} \alpha \rho w = 0, \\
 \partial_{t} \alpha \rho w + \partial_{x} (\alpha \rho w^{2} + \alpha p(\tau) ) = p(\tau) \partial_{x} \alpha,
\end{cases}
\end{equation}
where $\alpha$ is the cross section of the nozzle, $\rho$ is the density of the fluid, $\tau= \frac{1}{\rho}$ is the specific volume and $u$ is its velocity. The pressure law is classically supposed to be positive, decreasing and convex. 
As outlined in~\cite{CKS14}, system~\eqref{eq:nozzle} also describes the dynamics of flows in porous media. 
In that case $\alpha$ is the void fraction of the respective medium. 
In this application it is natural to consider a piecewise constant cross section
$$ \alpha= \alpha_{L} \mathbf{1}_{x<0} + \alpha_{R} \mathbf{1}_{x>0}, \ \alpha_{L}>0, \ \alpha_{R}>0, $$
while the derivation of~\eqref{eq:nozzle} for flows in a nozzle requires some smoothness on $\alpha$.

The difficulty is to define the left hand side $p(\tau) \partial_{x} \alpha= p(\tau) (\alpha_{R}-\alpha_{L}) \delta_{x=0}$ at point $x=0$. It is well known (\cite{IT92}, \cite{GLF04} and~\cite{MR3237558}) that system~\eqref{eq:nozzle} is not hyperbolic when one of the acoustic waves has speed $0$. The system is said to be resonant and uniqueness is lost. 
We will not tackle this problem here and restrict our attention for this particular example on subsonic flows.

For subsonic flows system~\eqref{eq:nozzle}, supplemented with the equation $\partial_{t} \alpha=0$, is strictly hyperbolic, and it is possible to show that
\begin{equation} \label{eq:Gnozzle}
 \alpha \rho w \text{ and } \frac{w^{2}}{2}+ e(\tau)+ \tau p(\tau) \text{ are continous at } x=0,
\end{equation}
where $\tau \mapsto e(\tau)$ is an antiderivative of $-p$. This is the set of interface conditions used at $x=0$ in model~\eqref{eq:nozzle}.

\section{Numerical method}
In the development of a numerical method the Riemann problem at the interface plays an important role.
Therefore we study the model~\eqref{eq:HypSysInter} with the initial conditions
\begin{equation}\label{eq:RP}
U^{0}(x)= U_{L} \mathbf{1}_{x<0}+ U_{R} \mathbf{1}_{x \geq 0}, \ (U_{L}, U_{R}) \in \Omega^{2} 
\end{equation}
and recall the strategy to prove existence and uniqueness in case of \eqref{eq: Gpart}. 
First, we look for a selfsimilar solution, i.e. a solution that can be written as $U(t,x)=W\left(\frac{x}{t}; U_{L}, U_{R} \right)$.
Thus the traces of the solution $U_{-}$ and $U_{+}$ on the lines $x=0^{-}$ and $x=0^{+}$ are constant in time.
Moreover, once these traces are determined, the whole solution is easily constructed by solving the Riemann problem without interface between $U_{L}$ and $U_{-}$ on the left half plane $x<0$, and between $U_{+}$ and $U_{R}$ on the right half plane $x>0$. 
For these states the following three conditions hold:
\begin{itemize}
\item On $x<0$, the solution coincides with the restriction of the solution of the Riemann problem between the left state $U_{L}$ and the right state $U_{-}$ for~\eqref{eq:HypSys}, and $U_{-}$ is the value of the solution on the line $x=0^{-}$, i.e.
\begin{equation}\label{eq:Uminus}
	U_- = W\left(0^-; U_{L}, U_{-} \right)\ .
\end{equation} 
\item On $x>0$, the solution coincides with the restriction of the solution of the Riemann problem between the left state $U_{+}$ and the right state $U_{R}$ for~\eqref{eq:HypSys}, and $U_{+}$ is the value of the solution on the line $x=0^{+}$, i.e. 
\begin{equation}\label{eq:Uplus}
	U_+ = W\left(0^+; U_{+}, U_{R} \right)\ .
\end{equation} 
\item $U_{-}$ and $U_{+}$ verify the interface conditions: 
\begin{equation} \label{eq:inG}
  (U_{-}, U_{+}) \in \mathcal{G}.
\end{equation}
\end{itemize}
Analytical results on the existence and uniqueness of a solution can be derived by assuming that $(U_{L},U_{R})$ is close enough from a stationary state for~\eqref{eq:HypSysInter}. Then, it is typically proved that the Riemann problem has a unique self-similar solution in the vicinity of this stationary solution. 
We refer the reader to~\cite{Bre00} for the case $n=1$ and to~\cite{CHS08} for $n=2$. 
See also~\cite{CG10} for the slightly different case where boundary conditions are imposed on $x=0$. 
In some particular cases, it is possible drop the smallness assumption on the initial data, as in~\cite{Agu15} for the model~\eqref{eq:Gcla} and in~\cite{LST08} for the Burgers-particle model.

This procedure can be used directly for the construction of a numerical scheme.
But the resulting Godunov method requires many details of the solution at the interface.  
Therefore it is only applicable for systems where the structure of the solution is known and it is not flexible for modifications of the considered equations.

In the following we describe a simple approach to approximate the solution at the interface, which does not require detailed information of the underlying Riemann problem.
\subsection{General setting}
Consider an equidistant spacial discretization of width $\Delta x$ and denote by $x_{j}= j \Delta x - \frac{\Delta x}{2}$ the centers of the cells. 
The point $x=0$ is located at the interface between the cells labelled  with $0$ and $1$. 
In time we consider the $n$-th time step $\Delta t^{n}= t^{n+1}-t^{n}$.

Away from this interface at $x=0$ any finite volume scheme based on the update formula
\begin{equation} \label{eq:updateFA}
U_{j}^{n+1}= U_{j}^{n}- \frac{\Delta t}{\Delta x}(f_{j+1/2}^{n}-f_{j-1/2}^{n})
\end{equation}
can be used.
The classical method can be applied as long as the stencil of the finite volume scheme (i.e. the cells $U_{k}^{n}$ used to computed $f_{j+1/2}^{n}$ and $f_{j-1/2}^{n}$) stays on one side of the interface at $x=0$.
In the sequel we focus on $2$-point fluxes $f_{j+1/2}^{n}= g(U_{j}^{n}, U_{j+1}^{n})$.
Thus we can use~\eqref{eq:updateFA} in all cells but have to define $f_{1/2}^{n,-}$  and $f_{1/2}^{n,+}$, the fluxes on the right and on the left of the interface at $x=0$.
Note that these fluxes do not coincide for a general choice of $\mathcal{G}$.

We follow a ghost cells approach and introduce at each time step $n$ the states $U_{-}^{n}$ and $U_{+}^{n}$ representing the fluid's properties at $x=0^{-}$ and $x=0^{+}$. 
These can be inserted into the numerical flux function such that the final numerical method is given by
\begin{equation} \label{eq:update}
\begin{cases}
U_{j}^{n+1}=U_{j}^{n}- \frac{\Delta t}{\Delta x}( g(U_{j}^{n}, U_{j+1}^{n})- g(U_{j-1}^{n}, U_{j}^{n})) & \text{ for } j \leq -1, \\
U_{0}^{n+1}=U_{0}^{n}- \frac{\Delta t}{\Delta x}( g(U_{0}^{n}, U_{-}^{n})- g(U_{-1}^{n}, U_{0}^{n})), \\
U_{1}^{n+1}=U_{1}^{n}- \frac{\Delta t}{\Delta x}( g(U_{1}^{n}, U_{2}^{n})- g(U_{+}^{n}, U_{1}^{n})), \\
U_{j}^{n+1}=U_{j}^{n}- \frac{\Delta t}{\Delta x}( g(U_{j}^{n}, U_{j+1}^{n})- g(U_{j}^{n}, U_{j-1}^{n})) & \text{ for } j \geq 2. 
\end{cases}
\end{equation}

\subsection{Choice of $U_{-}^{n}$ and $U_{+}^{n}$}
The key part of the numerical method is the choice of the values $U_{-}^{n}$ and $U_{+}^{n}$.
Their construction will depend on the numerical flux function $g$ used in \eqref{eq:update} since waves going to the left or right are incorporated differently.
As example consider the conservation of mass which is part of the coupling conditions \eqref{eq: Gpart}. 
If the approximation of $q_{\mp}$ at the interface does not imply that the first component of the numerical fluxes $g(U_{0}^{n}, U_{-}^{n})$ and $g(U_{+}^{n}, U_{1}^{n})$ coincide, mass will be lost or generated at the coupling point.

If for $g$ the Godunov flux $ g_{God}(U_{L}, U_{R})= f(W(0; U_{L}, U_{R}))$ is used, 
we can pick $U_{-}$ and $U_{+}$ as the traces of the exact solution of the Riemann problem such that~\eqref{eq:Uminus}, \eqref{eq:Uplus} and \eqref{eq:inG} hold. 
Note that \eqref{eq:Uminus} and \eqref{eq:Uplus} imply
\begin{equation} \label{eq:no wave}
g_{God}(U_{L},U_{-})= f(U_{-}) \quad \text{ and } \quad g_{God}(U_{+},U_{R})= f(U_{+}).  
\end{equation}

The aim of the present paper is to generalize the above procedure to cases when $g$ is an arbitrary numerical flux.
Unfortunately for many choices of $g$ the system~\eqref{eq:inG} and \eqref{eq:no wave} is over constrained and does not admit any solution. 
In~\cite{Bor15} this problem was relaxed by replacing \eqref{eq:no wave} with
\begin{equation} \label{eq:mini}
 \begin{aligned}
  (U_{-}^{n},U_{+}^{n}) &=  \argmin\Big( (\tilde{U}_{-}, \tilde{U}_{+}) \in \mathcal{G}, \\
    & \qquad |g(U_{0}^{n},\tilde{U}_{-})- f(\tilde{U}_{-})| + |g(\tilde{U}_{+},U_{1}^{n})- f(\tilde{U}_{+})| \Big). 
\end{aligned}
\end{equation}
In other words, the scheme tries to minimize the strength of the waves entering the junction. Indeed, we can rewrite Scheme~\eqref{eq:updateFA} in the fluctuation form
$$ U_{j}^{n+1}=U_{j}^{n}- \frac{\Delta t}{\Delta x} \big( (f_{j+1/2}^{n}-f(U_{j}^{n})) + (f(U_{j}^{n})-f_{j-1/2}^{n}) \big). $$
We interpret quantity $f_{j+1/2}^{n}-f(U_{j}^{n})$ as the overall contribution on the waves created at interface $x=x_{j+1/2}$ and entering the $j$-th cell (thus travelling to the left). Similarly, $f(U_{j}^{n})-f_{j-1/2}^{n}$ represents the contribution of the waves created at the left interface $x=x_{j-1/2}$ and going to the right. In~\eqref{eq:mini}, the numerical traces $(\tilde{U}_{-}, \tilde{U}_{+})$ are chosen such that the total strength of the waves entering the interface on its left (first term) and entering on its right (second term) is as small as possible.

In the present paper we explore a different strategy.
Although the exact solution only contains waves entering the domain, we allow waves inside the interface.
In the following we require that the numerical waves entering the coupling point cancel out at the interface.
Thus we look for $(U_{-}^{n},U_{+}^{n}) \in \mathcal{G}$ such that
\begin{equation} \label{eq:no fluct}
\underbrace{g(U_{0}^{n},U_{-}^{n})- f(U_{-}^{n})}_{\text{left entering waves}} + \underbrace{f(U_{+}^{n})-g(U_{+}^{n},U_{1}^{n})}_{\text{right entering waves}} =0
\end{equation}
holds.
Once system~(\ref{eq:inG}-\ref{eq:no fluct}) is solved, the fluid is updated with~\eqref{eq:update}, with the chosen numerical flux $g$.

One main advantage of the choice \eqref{eq:no fluct}, is that all quantities which are conserved by the coupling conditions, will be conserved exactly by the numerical method.
Note that this is not true neither with $(U_{-}^{n},U_{+}^{n})$ being the exact traces around the interface and $g$ another flux than the Godunov flux, nor with Choice~\eqref{eq:mini}.

\section{Numerical equilibrium} \label{S:Equilibrium}
One important aspect of a numerical method for coupling conditions is the ability to preserve numerically the equilibrium states of the underlying system.
Furthermore since the solution of the Riemann problem is self similar, equilibrium states will be generated at the interface.

As long as the numerical flux is consistent, i.e. $g(U,U)=f(U)$ $\forall U\in \Omega$, any two states $U_0$ and $U_1$ satisfying the coupling conditions admit $U_-=U_0$ and $U_+=U_1$ as solution when solving~(\ref{eq:inG}-\ref{eq:no fluct}).
Thus equilibrium states of the system \eqref{eq:HypSysInter} can be preserved by the numerical method.

More delicate is the reverse question, does the numerical method allow only equilibrium states of the system~\eqref{eq:HypSysInter} as stationary solutions.
Unfortunately the answer to this question is negative if the Godunov flux in combination with \eqref{eq:no fluct} is used. This is new compared to the scalar case, see~\cite{AC15}.
A detailed analysis of this case and a possible fix is presented in the section below.
In case of $g$ being the Rusanov flux the situation is simpler and no false equilibrium states can be obtained.

\subsection{The Godunov flux with $\mathcal G_\lambda$}
If $g$ is the Godunov flux $g_{God}$, than there exists a pair of constant states $(U_0, U_1)$ which does not satisfy the coupling conditions but is numerically a stationary solution. 
\begin{prop}~\label{P:Q2GodPart}
There exist $(U_{0},U_{1})$ that do not belong to $\mathcal{G}_{\lambda}$, for which there exists at least one couple of states $(U_{-},U_{+})$ such that
\begin{itemize}
\item $(U_{-},U_{+})$ verifies the interface conditions \eqref{eq:inG}. 
\item Equation~\eqref{eq:no fluct} is fulfilled. 
\item In the solution of the Riemann problem between $U_{0}$ and $U_{-}$, all the waves are going to the right, i.e. 
$g_{God}(U_{0}, U_{-})=f(U_{0}).$
\item In the solution of the Riemann problem between $U_{+}$ and $U_{1}$, all the waves are going to the left, i.e. 
$g_{God}(U_{+}, U_{1})=f(U_{1}).$
\end{itemize}
\end{prop}
\begin{proof}
Let us fix $U_{-}=(\rho_{-}, q_{-})$ and $U_{+}=(\rho_{+}, q_{+})$ such that $(U_{-}, U_{+})$ belongs to $\mathcal{G}_{\lambda}$. 
We have $q_{I} \de q_{-}=q_{+}$ and
$$ \left( \frac{q_{I}^{2}}{\rho_{-}} + c^{2} \rho_{-} \right) - \left( \frac{q_{I}^{2}}{\rho_{+}} + c^{2} \rho_{+} \right) = \lambda q_{I}. $$

We look for $U_{0}$ and $U_{1}$ verifying the last three conditions of the proposition, 
but not $(U_{0},U_{1}) \in \mathcal{G}_{\lambda}$.
By~\eqref{eq:no fluct}, we have $f(U_{0})-f(U_{1})=f(U_{-})-f(U_{+})$, thus $q_{0}=q_{1}:=q_{F}$ and
\begin{equation} \label{eq:NumEq}
  \left( \frac{q_{F}^{2}}{\rho_{0}} + c^{2} \rho_{0} \right) - \left( \frac{q_{F}^{2}}{\rho_{1}} + c^{2} \rho_{1} \right) = \lambda q_{I}.
\end{equation}
The fact that $(U_{0}, U_{1})$ does not necessarily belong to $\mathcal{G}_{\lambda}$ follows from the fact that $q_{I}$ and $q_{F}$ can be different.

Let us now focus on the condition $g_{God}(U_{0},U_{-})=f(U_{0})$.
We recall that this condition exactly states that in the Riemann problem between $U_{0}$ and $U_{-}$ the waves only go to the right. 
In the case of the isothermal Euler equations and for a fixed $U_{-}$, it is possible to describe the set of such $U_{0}$, see~\cite{Agu14} for details. 
In the $(\rho,q)$-plane, it consists of the union of an increasing curve $\Gamma_{\rightarrow}^{sub}(U_{-})$ included in the subsonic triangle $\{(\rho, q): |q| < c \rho \}$ and an open set $\Omega_{\rightarrow}^{sup}(U_{-})$ included in $\{(\rho, q), q > c \rho \}$.
These sets are shown with blue colour in Figure~\ref{F:SubSub}. 
The important points are that the states in $\Gamma_{\rightarrow}^{sub}(U_{-})$ are linked to $U_{-}$ by only a $2$-wave and that  $\Omega_{\rightarrow}^{sub}(U_{-})$ is delimited by a curve which is the image of $\Gamma_{\rightarrow}^{sub}(U_{-})$ under the operation ``stationary shock''
$$ 
\begin{array}{ccc}
\Gamma_{\rightarrow}^{sub}(U_{-}) \cap \{q>0\} & \longrightarrow & \partial \Omega_{\rightarrow}^{sup}(U_{-}) \\
(\rho, q)  &  \mapsto & \left(\frac{q^{2}}{c^{2} \rho}, q \right)\ .
\end{array}
$$
Similarly, the set of all $U_{1}$ such that $g_{God}(U_{+},U_{1})=f(U_{1})$ is the union of a decreasing curve $\Gamma_{\leftarrow}^{sub}(U_{+})$ included in the subsonic triangle and of an open set $\Omega_{\leftarrow}^{sup}(U_{+})$ included in $\{(\rho, q), q<-c \rho \}$.
 These sets are depicted in red in Figure~\ref{F:SubSub}.

\begin{figure}[h!t]
\centering
\begin{psfrags}
\psfrag{A}[][][1.4]{$\Omega^{sup}_{\rightarrow}(U_{-})$}
\psfrag{B}[][][1.4]{$\Gamma^{sub}_{\rightarrow}(U_{-})$}
\psfrag{C}[][][1.4]{$\Omega^{sup}_{\leftarrow}(U_{+})$}
\psfrag{D}[][][1.4]{$\Gamma^{sub}_{\leftarrow}(U_{+})$}
\psfrag{a}[][][1.2]{$q$}
\psfrag{r}[][][1.2]{$\rho$}
\psfrag{q}[][][1.2]{$q_{F}$}
\psfrag{qc}[][][1.2]{$q_{I}$}
\psfrag{URc}[][][1.2]{$U_{+}$} 
\psfrag{ULc}[][][1.2]{$U_{-}$} 
\psfrag{UR}[][][1.2]{$U_{1}$} 
\psfrag{UL}[][][1.2]{$U_{0}$} 
\psfrag{tV0}[][][1.2]{$\tilde{V_{0}}$} 
\psfrag{ULt}[][][1.2]{$V_{0}$} 
\psfrag{U0}[][][1.2]{$U_{0}$} 
\psfrag{H}[][][1.2]{$bla$} 
\includegraphics[width=\linewidth, height = 0.9\linewidth]{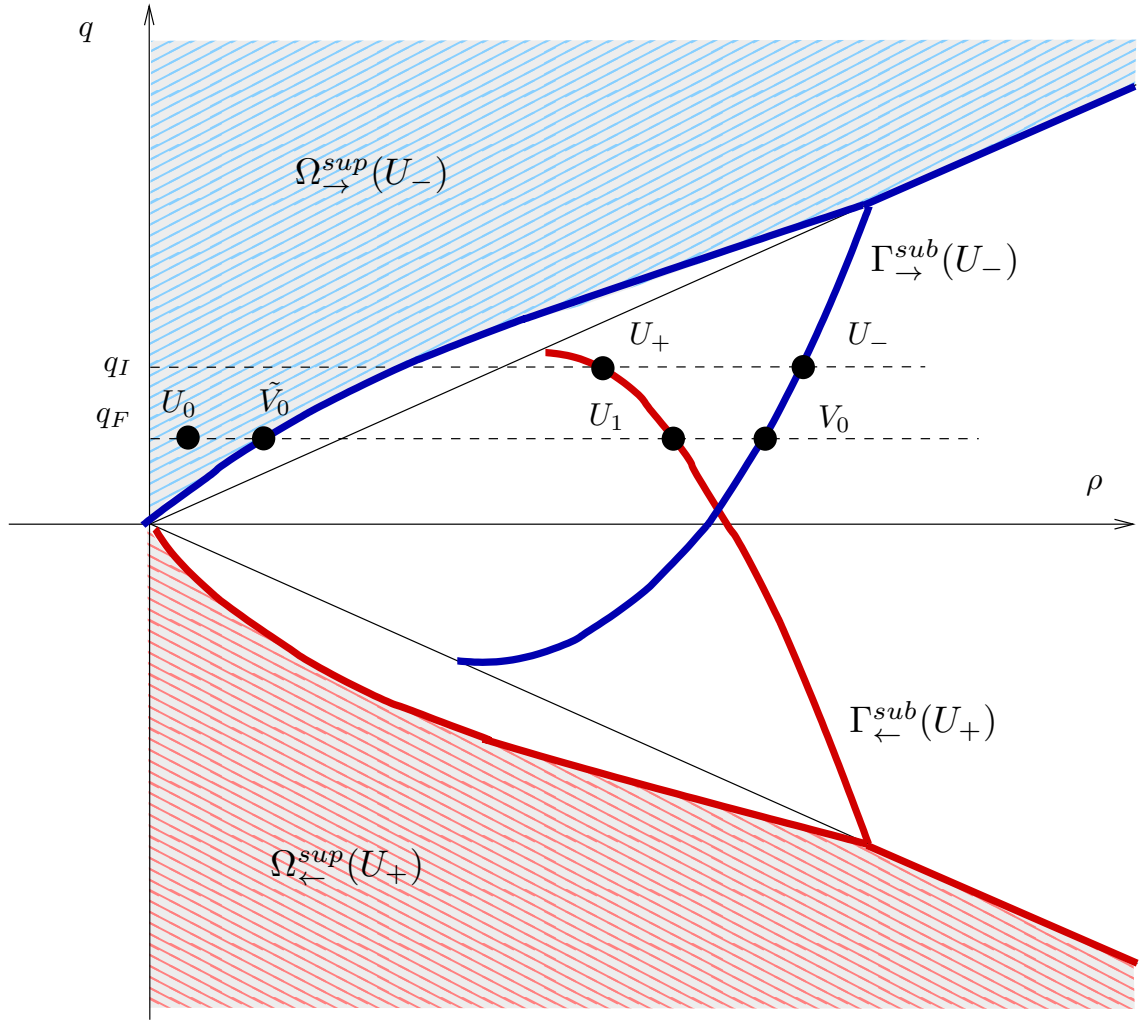}
\caption{How to find $U_{0}$ and $U_{1}$ verifying the conditions of Proposition~\ref{P:Q2GodPart} from $U_{-}$ and $U_{+}$}
\label{F:SubSub}
\end{psfrags}
\end{figure}
We are now in position to find all couples $(U_{0},U_{1})$ verifying the last three conditions of the proposition
, in the case where both $U_{-}$ and $U_{+}$ are subsonic.
Without loss of generality we suppose that $q_{I}$ is non negative.
Then, we fix $q_{F}$ in the interval $(0, q_{I})$, and we denote by $U_{1}=(\rho_{1}, q_{F})$ the state at the intersection of $\Gamma^{sub}_{\leftarrow}(U_{+})$ and $q=q_{F}$.
Similarly is $V_{0}=(r_{0}, q_{F})$ at the intersection of $\Gamma^{sub}_{\rightarrow}(U_{-})$ and $q=q_{F}$. 
As $r_{0}<\rho_{-}$, $V_{0}$ is linked to $U_{-}$ by a $2$-rarefaction wave with positive speed and we have
$$ \frac{q_{F}^{2}}{r_{0}}+c^{2} r_{0} < \frac{q_{I}^{2}}{\rho_{-}}+c^{2} \rho_{-}\ . $$
On the right hand side is $\rho_{+}< \rho_{1}$, thus $U_{+}$ is linked to $U_{1}$ by a $1$-shock with negative speed and
$$ \frac{q_{F}^{2}}{\rho_{1}}+c^{2} \rho_{1} > \frac{q_{I}^{2}}{\rho_{+}}+c^{2} \rho_{+}. $$
It follows that
$$ \left( \frac{q_{F}^{2}}{r_{0}}+c^{2} r_{0} \right) - \left( \frac{q_{F}^{2}}{\rho_{1}}+c^{2} \rho_{1} \right) <\left( \frac{q_{I}^{2}}{\rho_{-}}+c^{2} \rho_{-} \right) - \left( \frac{q_{I}^{2}}{\rho_{+}}+c^{2} \rho_{+}\right) = \lambda q_{I}. $$
Now, denote by $\tilde{V_{0}}= \left( \tilde{r_{0}}=\frac{q_{F}^{2}}{c^{2} r_{0}}, q_{F} \right)$ the states linked to $V_{0}$ by a stationary shock. 
A simple computation yields $\frac{q_{F}^{2}}{\tilde{r_{0}}}+c^{2} \tilde{r_{0}}=\frac{q_{F}^{2}}{r_{0}}+c^{2} r_{0}$ and $\tilde{r_{0}} \leq \frac{q_{F}}{c}$. 
Moreover, the function 
$ \rho \mapsto \frac{q_{F}^{2}}{\rho}+c^{2} \rho $
is decreasing on $\left(0, \frac{q_{F}}{c} \right)$, and tends to $+ \infty$ when $\rho$ tends to $0$. 
Thus, there exists a unique $\rho_{0}$ smaller than $\tilde{r_{0}}$ such that
$$ \left( \frac{q_{F}^{2}}{\rho_{0}}+c^{2} \rho_{0} \right) - \left( \frac{q_{F}^{2}}{\rho_{1}}+c^{2} \rho_{1} \right) = \lambda q_{I} $$
and the state $U_{0}=(\rho_{0}, q_{F})$ belongs to $\Omega^{sup}_{\rightarrow}(U_{-})$ and verifies~\eqref{eq:NumEq}. 
Summarizing the above construction, the two pairs $(U_{0},U_{1})$ and $(U_{-},U_{+})$ verify the four points of the proposition
, and since 
$q_{I}\neq q_{F}$, $(U_{0},U_{1})$ 
does not belong to $\mathcal{G}_{\lambda}$.
\end{proof}

\begin{ex}
Take $ c=1, \ \lambda=1$,
$$U_{L}= ( \rho_{L}= 0.109272, \ q_{F}=0.618826), \ U_{R}=( \rho_{R}= 3.024454, q_{F}) $$
and
$$U_{-}=( \rho_{-}= 3.31851, \ q_{I}=0.771179) \ \text{ and } \ U_{+}=( \rho_{+}= 2.824455, q_{I}). $$
Then the pair of states $(U_{L}, U_{R})$ does not belong to $\mathcal{G}_{\lambda}$, i.e.
$$ \left( \frac{q_{F}^{2}}{\rho_{L}} + c^{2} \rho_{L} \right) - \left( \frac{q_{F}^{2}}{\rho_{R}} + c^{2} \rho_{R} \right) \neq \lambda q_{F}$$
but $(U_{-}, U_{+})$ does. 
Moreover, it holds
$$ g_{God}(U_{0},U_{-})= f(U_{0}) \ \text{ and } \ g_{God}(U_{+},U_{1})= f(U_{1}). $$
This is an example of a numerical equilibrium state which is not related to an exact stationary state as shown in Proposition~\ref{P:Q2GodPart}.
\end{ex}

\begin{rem}
Uniqueness is easily restored by replacing $q_{I}$ by $q_{F}$ in~\eqref{eq:NumEq}, i.e. by imposing
$$ \left( \frac{q_{F}^{2}}{\rho_{0}} + c^{2} \rho_{0} \right) - \left( \frac{q_{F}^{2}}{\rho_{1}} + c^{2} \rho_{1} \right) = \lambda q_{F}$$
instead of~\eqref{eq:NumEq}.
Numerically, and denoting by $(g^{\rho}, g^{q})$ the two components of the numerical flux, it boils down to search $U_{-}=(\rho_{-}, q_{I})$  and $U_{+}=(\rho_{+}, q_{I})$  such that
$$ 
\begin{cases}
  g^{\rho}(U_{0},U_{-})= g^{\rho}(U_{+}, U_{1}), \\
  g^{q}(U_{0},U_{-}) - g^{q}(U_{+}, U_{1}) = \lambda g^{\rho}(U_{0},U_{-}),
\end{cases}
$$
instead of
$$ 
\begin{cases}
  g^{\rho}(U_{0},U_{-})= g^{\rho}(U_{+}, U_{1}), \\
  g^{q}(U_{0},U_{-}) - g^{q}(U_{+}, U_{1}) = \lambda q_{I}.  
\end{cases}
$$
This fix reflects the fact that we actually want to impose is~\eqref{eq:no wave} and not just~\eqref{eq:no fluct}.
\end{rem}

\subsection{The Godunov flux with $\mathcal G_0$}
	In the classical case, i.e. the case $\lambda = 0$, the above construction gives a number of couples of states $(U_{0},U_{1})$ satisfying the conditions of Proposition~\ref{P:Q2GodPart}.
	However this time, Equation~\eqref{eq:NumEq} reduces to the equality of the mass flux $\frac{q^{2}}{\rho}+c^{2} \rho$ through the interface, thus $(U_{0},U_{1})$ verifies automatically the interface conditions.

\subsection{The Rusanov flux with any interface conditions}\label{Q: Rus cla}
As alternative to the Godunov flux we consider the Rusanov flux $g_{Rus}$ given by
\begin{equation} \label{eq:fluxRus}
 g_{Rus}(U_{L},U_{R})= \frac{f(U_{L})+f(U_{R})}{2}-\frac{A}{2}(U_{R}-U_{L}),
\end{equation}
where $A$ verifies the subcharacteristic condition
 \begin{equation} \label{eq: subchar1}
 A \geq \max \left( \frac{|q_{R}|}{\rho_{R}}, \frac{|q_{L}|}{\rho_{L}}, \frac{|q_{\star}|}{\rho_{\star}} \right)+ c
\end{equation}
and 
\begin{equation} \label{eq:U*Rus}
 U_{\star}= \frac{U_{L}+U_{R}}{2}- \frac{1}{2A}(f(U_{R})-f(U_{L})). 
\end{equation}
The Rusanov flux is one of the simplest flux functions possible.  
Therefore we consider it as an easy alternative to the complicated Godunov method and  a representative of many central schemes.
\begin{rem}
  The Rusanov flux can be viewed as a HLL approximate Riemann solver with wave speeds $-A$ and $A$, see e.g.~\cite{GR96}.  
 The conservation of the density and the momentum yield as intermediate state 
 $$(\rho_{\star}, q_{\star})= \left( \frac{\rho_{L} + \rho_{R}}{2} + \frac{q_{L}-q_{R}}{2A}, \frac{q_{L} + q_{R}}{2} + \frac{\eta_{L}-\eta_{R}}{2A} \right),$$ 
which is the expression of the middle state $U_{*}$ given in~\eqref{eq:U*Rus}.
\end{rem}

We check that, independently of the coupling conditions, stationary solutions are uniquely determined.
\begin{prop}
If the pair of state $(U_{0},U_{1})$ is such that for some $(U_{-},U_{+})$ in $\mathcal{G}$ the following holds
\begin{equation} \label{eq:RusEq}
\begin{cases}
g_{Rus}(U_{0}, U_{-})=f(U_{0})\ , \\
g_{Rus}(U_{+}, U_{1})=f(U_{1})\ .
\end{cases}
\end{equation}
Then $U_{0}=U_{-}$ and $U_{1}=U_{+}$ and in particular $(U_{0},U_{1})$ belongs to $\mathcal{G}$.
\end{prop}
\begin{proof}
Equation~\eqref{eq:RusEq} we can rewrite as
$$ 
\begin{cases}
f(U_{-})-f(U_{0})= A(U_{-}-U_{0}), \\
f(U_{+})-f(U_{1})=A(U_{1}-U_{+}).
\end{cases}
$$
Thus $U_{0}$ is linked to $U_{-}$ by a shock at speed $A$, and $U_{+}$ is linked to $U_{1}$ by a shock at speed $-A$, with contradicts  Condition~\eqref{eq: subchar1} on $A$, unless $U_{-}=U_{0}$ and $U_{1}=U_{+}$.
\end{proof}

\section{Analysis of the Rusanov flux} \label{S:Rusanov}
In this section we analyze if the coupling procedure with the Rusanov flux has a unique solution, for the interface conditions $\mathcal{G}_{\lambda}$.
Unfortunately this is not always the case, but we will develop criteria to single out the correct interface values.
Therefore we investigate the evolution of the entropy when using the Rusanov flux.
\begin{prop} \label{p: EI Rus}
If condition~\eqref{eq: subchar1} holds, then we have the entropy inequality
\begin{equation} \label{eq: EI Rus}
 F(U_{R})- F(U_{L}) \leq A \big(E(U_{R})+ E(U_{L})-2E(U_{*})\big)
\end{equation}
with equality if and only if $U_{R}=U_{L}$.
\end{prop}
\begin{proof} This is a straightforward consequence of the Jensen inequality applied on the underlying approximate Riemann solver, see~\cite{HLL83}.
\end{proof}

\subsection{The Rusanov flux with $\mathcal G_0$}
Consider the classical coupling conditions given by~\eqref{eq:Gcla} and denote the momentum flux (also called charge) $\frac{q^2}{\rho}+c^2 \rho$ by $\eta$.

\begin{prop} \label{p: existence cla} Let $U_{L}= (\rho_{L},q_{L})$ and $U_{R}=(\rho_{R},q_{R})$  be two elements of $\Omega$, and consider $U_{*}=(\rho_{*},q_{*})$ the state defined by~\eqref{eq:U*Rus}. 
Then system~(\ref{eq:inG}-\ref{eq:no fluct}) with $\mathcal{G}= \mathcal{G}_{0}$ and $g=g_{Rus}$
 always admits the solution
 $$U_{-}=U_{+}=(\rho_{\star}, q_{\star}). $$
\begin{itemize}
 \item
 This solution is unique  if $\rho_{\star}^2 - \frac{q_{\star}^2}{c^2} \leq 0$ or $q_{\star}=0$. 
 \item If $\rho_{\star}^2 - \frac{q_{\star}^2}{c^2} > 0$ System~(\ref{eq:inG}-\ref{eq:no fluct}) has a second solution, namely
\begin{equation} \label{eq: sol noncla}
 \begin{cases}
 U_{-}= (\rho_{\star}- r, q_{\star}) \text{ and } U_{+}= (\rho_{\star}- r, q_{\star}) \ \text{ if } \ q_{\star}>0, \\
 U_{-}= (\rho_{\star}+ r, q_{\star}) \text{ and } U_{+}= (\rho_{\star}- r, q_{\star}) \ \text{ if } \ q_{\star}<0,
\end{cases}
\end{equation}
 where $r=\sqrt{\rho_{\star}^2 - \frac{q_{\star}^2}{c^2} }$.
\end{itemize}
\end{prop}

\begin{proof}
The interface conditions~\eqref{eq:Gcla} imply $q_{-}=q_{+}:=q$ and $\eta_{-}=\eta_{+}:=\eta$. 
Using~\eqref{eq:no fluct} yields the equality of the numerical flux on both sides of the particle, i.e.
$$ 
\begin{cases}
     q+q_{L} -A(\rho_{-} -\rho_{L}) = q_{R}+q -A(\rho_{R} -\rho_{+}), \\
     \eta+\eta_{L} -A(q -q_{L}) = \eta_{R}+\eta -A(q_{R} -q).
\end{cases}
$$
these we can rearrange to obtain
\begin{equation} \label{eq: q cla}
 q=q_{\star}=\frac{q_{L} + q_{R}}{2} + \frac{\eta_{L}-\eta_{R}}{2A}
\end{equation}
and
$$ \frac{\rho_{-} + \rho_{+}}{2}=\frac{\rho_{L} + \rho_{R}}{2} + \frac{q_{L}-q_{R}}{2A}=\rho_{\star}. $$
We express $\rho_{-}$ and $\rho_{+}$ by $\rho_{-}=\rho_{\star}-\tilde{r}$ and $\rho_{+}=\rho_{\star}+\tilde{r}$, for some real $\tilde{r}$ in $(-\rho_{\star},\rho_{\star})$. 
The interface condition associated to the conservation of momentum  $\eta_{-}=\eta_{+}$ gives 
$$  \frac{q_{\star}^2}{\rho_{\star}-\tilde{r}} + c^2 (\rho_{\star}-\tilde{r}) = \frac{q_{\star}^2}{\rho_{\star}+\tilde{r}} + c^2 (\rho_{\star}+\tilde{r})\ ,$$
which always admits the trivial solution $\tilde{r}=0$. 
If $|q_{\star}|$ and $\rho_{\star}^2 - \frac{q^2}{c^2}$ are both strictly positive, there are two other solutions in $(-\rho_{\star},\rho_{\star})$ given by
$$  \tilde{r} =  \pm \sqrt{\rho_{\star}^2 - \frac{q_{\star}^2}{c^2}} = \pm r. $$

It remains to see which of the two solutions is selected by the entropy condition $F(U_{-}) \leq F(U_{+})$. On the one hand, the coupling conditions~\eqref{eq:Gcla} are nothing but the Rankine--Hugoniot relations for a stationary shock. This shock is entropy satisfying if and only if 
$$ 
\begin{cases}
  u_{-} \geq c \ \text{ and }  \ u_{+} \leq c & \text{ if } q>0, \\
  u_{+} \leq -c \ \text{ and }  \ u_{-} \geq -c & \text{ if } q<0.
\end{cases}
$$
On the other hand, it is easy to show that
$$ 
\begin{cases}
  \frac{q_{\star}}{\rho_{\star}- r} \geq c \ \text{ and } \frac{q_{\star}}{\rho_{\star}+ r} \leq c & \text{ if } q_{\star} >0, \\
  \frac{q_{\star}}{\rho_{\star}- r} \leq -c \ \text{ and } \frac{q_{\star}}{\rho_{\star}+ r} \geq -c & \text{ if } q_{\star}<0.
\end{cases}
$$
As a consequence, the other entropy satisfying solution is $\rho_{-}= \rho_{\star}- \sign(q_{\star}) r$ and $\rho_{+}= \rho_{\star}+ \sign(q_{\star}) r$.
\end{proof}

We now investigate the existence of an entropy inequality for the scheme~\eqref{eq:update} with the Rusanov flux. 
\begin{prop} \label{p: EInum1}
Consider the scheme~\eqref{eq:update} with $U_{-}^{n}=U_{+}^{n}=(\rho_{\star}, q_{\star})$ defined as in Proposition~\ref{p: existence cla}, with $U_{L}=U_{0}^{n}$ and $U_{R}=U_{1}^{n}$. 
Suppose that $A$ verifies the subcharacteristic condition~\eqref{eq: subchar1} for ($U_{L}=U_{j}^{n}$ and $U_{R}=U_{j+1}^{n}$) for all $j \neq 0$, for ($U_{L}=U_{0}^{n}$ and $U_{L}=U_{-}^{n}$), and for ($U_{L}=U_{+}^{n}$ and $U_{L}=U_{1}^{n}$) .

Then, the scheme verifies the discrete entropy inequality
\begin{equation} \label{eq:EInum}
  E(U_{j}^{n+1}) \leq E(U_{j}^{n})- \frac{\Delta t^{n}}{\Delta x} (F_{j+1/2,-}^{n}-F_{j-1/2,+}^{n}) \ ,
\end{equation}
where
$$ 
\begin{cases}
\displaystyle F_{j+1/2,-}^{n}= \frac{F(U_{j}^{n})+ F(U_{j+1}^{n})}{2}- \frac{A}{2}(E(U_{j+1}^{n})- E(U_{j-1}^{n})) & \text{ if } j \neq 0, \\
F_{j+1/2,-}^{n}=F_{j+1/2,+}^{n} & \text{ if } j \neq 0, \\
 \displaystyle F_{1/2,-}^{n}= \frac{F(U_{0}^{n})+ F(U_{-}^{n})}{2}- \frac{A}{2}(E(U_{-}^{n})- E(U_{0}^{n})), \\
 \displaystyle F_{1/2,+}^{n}= \frac{F(U_{+}^{n})+ F(U_{1}^{n})}{2}- \frac{A}{2}(E(U_{1}^{n})- E(U_{+}^{n})).
\end{cases}
$$
Moreover, $F_{1/2,-}^{n} \geq F_{1/2,+}^{n}$, thus in particular
\begin{equation} \label{eq:dissipentrop}
 E(U_{0}^{n+1}) + E(U_{1}^{n+1}) \leq E(U_{0}^{n})+ E(U_{1}^{n})- \frac{\Delta t^{n}}{\Delta x} (F_{3/2,-}^{n}-F_{-1/2,+}^{n}).
\end{equation}
\end{prop}

\begin{proof} Let us first recall that the result away from the particle ($j \neq 0$) follows from the interpretation of the Rusanov scheme~\eqref{eq:fluxRus} as an approximate Riemann solver with wave speeds $-A$ and $A$, see Proposition~\ref{eq: EI Rus}.
With that interpretation, inequality~\eqref{eq: EI Rus} yields~\eqref{eq:EInum}, see~\cite{GR96}. For $j=0$, we just have to distinguish between the approximate Riemann solver used on the left of the interface, which corresponds to $U_{L}=U_{0}^{n}$ and $U_{R}=U_{-}^{n}$ and yields~\eqref{eq:EInum} for $j=0$, and the one used on the right of the interface, for which $U_{L}=U_{+}^{n}$ and $U_{R}=U_{1}^{n}$ and which yields~\eqref{eq:EInum} for $j=1$. 

It remains to prove that $F_{1/2,-}^{n} \geq F_{1/2,+}^{n}$. We start with the classical entropy relation~\eqref{eq: EI Rus} with $U_{L}=U_{0}^{n}$ and $U_{R}=U_{1}^{n}$
 $$F(U_{0}^{n})-F(U_{1}^{n}) \leq -A(E(U_{*}^{n})-E(U_{0}^{n}))+A(E(U_{1}^{n})-E(U_{*}^{n})). $$
 Introducing $F(U_{*}^{n})$ in the left hand side, and reorganizing the inequality, we obtain
$$ 
\begin{aligned}
  \frac{F(U_{1}^{n})+F(U_{*}^{n})}{2}& -\frac{A}{2} \big(E(U_{1}^{n})-E(U_{*}^{n})\big) \\
  	&\leq \frac{F(U_{0}^{n})+F(U_{*}^{n})}{2}  -\frac{A}{2} \big(E(U_{*}^{n})-E(U_{0}^{n}) \big) 
\end{aligned}
$$
and the result.
\end{proof}

We now state a similar property for the other solution in Proposition~\ref{p: existence cla}.
\begin{prop} Consider the scheme~\eqref{eq:update} with 
$$U_{-}^{n}=(\rho_{*} - \sign(q_{*}) r, q_{\star}) \ \text{ and } \ U_{+}^{n}=(\rho_{*} + \sign(q_{*}) r, q_{\star}) $$
defined (if possible) by the second point of Proposition~\ref{p: existence cla} with $U_{L}=U_{0}^{n}$ and $U_{R}=U_{1}^{n}$.
Suppose that $A$ verifies the same subcharacteristic condition than in Proposition~\ref{p: EInum1}.
Then~\eqref{eq:EInum} holds. Moreover if the underlying approximate Riemann solver verifies the entropy inequality
\begin{equation} \label{eq: entrop 1/2}
\begin{aligned}
  F(U_{1})-F(U_{0}) \leq A(E(U_{0})+E(U_{1})-E(U_{-})-E(U_{+})), 
\end{aligned} 
\end{equation}
then $F_{1/2,-}^{n} \geq F_{1/2,+}^{n}$ and~\eqref{eq:dissipentrop} holds.
\end{prop}
\begin{proof}
The fact that~\eqref{eq:EInum} holds is obtained as before. Indeed, as the subcharacteristic condition~\eqref{eq: EI Rus} with $U_{L}=U_{0}^{n}$ and $U_{R}=U_{-}^{n}$, we have
$$ F(U_{-}^{n})-F(U_{0}^{n}) \leq A\big(E(U_{-}^{n})+ E(U_{*0}^{n})- 2 E(U_{*0}^{n}) \big), $$
where $U_{*0}^{n}= \frac{U_{L}^{n}+U_{-}^{n}}{2}- \frac{1}{2A}(f(U_{-}^{n})-f(U_{0}^{n}))$, which yields~\eqref{eq:EInum} for $j=0$ (see once again~\cite{GR96}). 
The reasoning is similar on the cell $1$, because the subcharacteristic condition implies that $A$ is larger than $|u_{+}|+c$. 
The last part of the proposition follows from the definitions of $F_{1/2, \pm}^{n}$, and the fact that $F(U_{+}^{n}) \leq F(U_{-}^{n})$.
\end{proof}

We now check that the scheme is consistent, i.e. that if $U_{L}=U_{R}$, the only possible choice in Proposition~\ref{p: existence cla} is $U_{-}=U_{+}=U_{R}=U_{L}$.
\begin{prop}
 If $U_{L}=U_{R}$, the second solution of~\ref{p: existence cla} does not verify inequality~\eqref{eq: entrop 1/2}.
\end{prop}
\begin{proof} Let us first notice that if $U_{L}=U_{R}$, then $U_{*}=U_{L}=U_{R}$.
 We have to eliminate the solution $U_{-}=U_{*} \pm (r, 0)$ and $U_{+}=U_{*} \mp (r,0)$ when it exists. 
 The intermediate state in the Rusanov approximate Riemann solver between the left state $U_{+}$ and the right state $U_{-}$ is 
 $$ \frac{U_{-}+U_{+}}{2} - \frac{1}{2A}(f(U_{-})-f(U_{+})) = U_{L}\ .$$
 By definition of~$\mathcal{G}_{0}$ we have equality of the fluxes $f(U_{+})=f(U_{-})$ 
 and under the subcharacteristic condition~\eqref{eq: subchar1}, we obtain
  $$ 
\begin{aligned}
  F(U_{-})-F(U_{+}) &\leq -A(E(U_{L})-E(U_{+}) + A(E(U_{-})-E(U_{L}))  \\
  	& \leq A (E(U_{+})+ E(U_{-}) - 2 E(U_{L}))\ .
\end{aligned}
 $$
Thus, we have equality in~\eqref{eq: entrop 1/2}, which by Proposition~\ref{p: EI Rus} only holds if $U_{-}=U_{+}$.
\end{proof}

Next we prove that the scheme is exact on isolated admissible stationary shocks, which are the piecewise constant equilibrium states associated to $\mathcal{G}_{0}$.
\begin{prop} \label{p: exashock} Consider a Riemann problem~\eqref{eq:RP}, where  $U_{L}$ and $U_{R}$ are such that
 $$ q_{L}=q_{R}\neq 0, \ u_{L}>u_{R}, \  \rho_{L} \neq \rho_{R} \text{ and } \ \rho_{L} \rho_{R}= \frac{q_{L}q_{R}}{c^{2}}. $$
 The exact solution of this Riemann problem is
 $$U(x,t)= U_{L} \mathbf{1}_{x<0} + U_{R} \mathbf{1}_{x>0}. $$
 Then for all $n \in \N$, 
 $$U_{j}^{n}= 
\begin{cases}
  U_{L} & \text{ if } j \leq 0, \\
   U_{R} & \text{ if } j \geq 1.
\end{cases}
 $$
\end{prop}

\begin{proof}
 To prove that the solution remains constant we have to check that $U_{1/2,-}^{0}=U_{L}$, $U_{1/2,+}^{0}=U_{R}$ and that~\eqref{eq: entrop 1/2} holds.
  First we note that (with the notation of Proposition~\ref{p: existence cla})
 $ \rho_{\star}= \frac{\rho_{L}+\rho_{R}}{2} \ \text{ and } q_{\star}=q_{L}=q_{R}. $
 Thus $\rho_{\star}$ is larger than $|q_{\star}|/c$ if and only if $\frac{\rho_{L}+ \frac{q_{\star}^{2}}{c^{2}\rho_{L}}}{2} \geq \frac{|q_{\star}|}{c}$. 
 An elementary computation shows that $\rho \mapsto \rho + \frac{q_{\star}^{2}}{c^{2}\rho}$ reaches its only minimum $\frac{2 |q_{\star}|}{c}$ for $\rho= \frac{|q_{\star}|}{c}$, and the result follows. 
 Thus the second solution of Proposition~\ref{p: existence cla} exists, and it is easy to check that $U_{-}=U_{L}$ and $U_{+}=U_{R}$. 
 Furthermore, Equation~\eqref{eq: entrop 1/2} reduces to $F(U_{R}) \leq F(U_{L})$, which is true for admissible shocks (it is shown in the proof of Proposition~\ref{p: existence cla}).
\end{proof}

\subsection{The Rusanov flux with $\mathcal G_\lambda$}
 \label{Q: Rus part}
Now we investigate the interface conditions~\eqref{eq: Gpart} for the fluid / particle coupling.
We begin with an observation analogue to Proposition~\ref{p: existence cla}.
\begin{prop} \label{existence lambda} Let $\tilde{\mathcal{G}}_{\lambda}$ be the set
$$ \label{eq:Gparttilde}
\tilde{ \mathcal{G}}_{\lambda} = 
\begin{pmatrix} 
q_{-}=q_{+}=: q \\
\left( \frac{q^{2}}{\rho_{-}}+c^{2} \rho_{-} \right) -  \left( \frac{q^{2}}{\rho_{+}}+c^{2} \rho_{+} \right) = \lambda q 
\end{pmatrix}.
$$
Then for $g=g_{Rus}$ and every $(U_{L},U_{R})$ in $(\R_{+} \times \R)^2$, there exists at least one and at most three solutions $(U_{-},U_{+})$ of~\eqref{eq:no fluct} in $\tilde{ \mathcal{G}}_{\lambda}$. 
\end{prop}

\begin{proof}
The first interface condition implies that the mass is conserved across the particle $q_{-}=q_{+}=:q$. 
The second equation quantifies the loss of momentum $\eta_{-} - \eta_{+}=\lambda q$. 
System~\eqref{eq:no fluct} is designed as the numerical counterpart of those interface conditions. 
It says that the mass fluxes at the interface are equal (numerical conservation of the density) and quantifies the jump of the momentum flux. 
For the Rusanov flux these conditions are
$$ 
\begin{cases}
     q+q_{0} -A(\rho_{-} -\rho_{0}) = q_{1}+q -A(\rho_{1} -\rho_{+}), \\
     \frac{1}{2}\big(\eta_{-}+\eta_{0} -A(q -q_{0})\big)- \frac{1}{2}\big(\eta_{1}+\eta_{+} -A(q_{1} -q)\big)= \eta_{-}-\eta_{+} .
\end{cases}
$$
The second line yields
$$\eta_{0}- \eta_{1} - 2 A q + A(q_{0}+q_{1})= \eta_{-}- \eta_{+}= \lambda q$$
and thus
$$ q= \frac{A(q_{0}+q_{1})}{\lambda+2A}+ \frac{\eta_{0}-\eta_{1}}{\lambda+2A}. $$
Note that if $\lambda=0$, we recover the classical case~\eqref{eq: q cla}, while if $((\rho_{0},q_{0}),(\rho_{1},q_{1}))$ belongs to $\mathcal{G}_{\lambda}$, we recover $q=q_{0}=q_{1}$.
Concerning the first equation, we obtain as before
$$ \frac{\rho_{-} + \rho_{+}}{2}=\frac{\rho_{0} + \rho_{1}}{2} + \frac{q_{0}-q_{1}}{2A}=:\rho_{\star}. $$
Thus we can express $\rho_{-}=\rho_{\star}-r$ and $\rho_{+}=\rho_{\star}+r$ for some real $r$ in $(-\rho_{\star},\rho_{\star})$. 
The interface condition associated to the momentum equation  $\eta_{-}-\eta_{+}= \lambda q$ gives
\begin{equation} \label{eq: r part}
  \left(\frac{q^2}{\rho_{\star}-r} + c^2 (\rho_{\star}-r) \right) - \left(\frac{q^2}{\rho_{\star}+r} + c^2 (\rho_{\star}+r) \right) = \lambda q.
\end{equation}
Interpreting this as a function of $r$, the left hand side tends to $+ \infty$ when $r$ tends to $\rho_{\star}$ and to $- \infty$ when $r$ tends to $-\rho_{\star}$. 
Thus, Equation~\eqref{eq: r part} admits at least one solution in $(-\rho_{\star}, \rho_{\star})$, and at most three as it is equivalent to find the roots of a third degree polynomial expression.
\end{proof}

The evolution of the roots of~\eqref{eq: r part} is depicted in Figure~\ref{F: Roots} for two different initial data. 
For $\rho_{L}=\rho_{R}= 5$ and $q_{L}=q_{R}=2.5$, $c=1$ and $\lambda = 0$ there are three roots. As $\lambda$ becomes larger, the root becomes unique. The root is unique for every choice of $\lambda$ when $\rho_{L}=\rho_{R}= 1$ and $q_{L}=q_{R}=2.5$.
Note that in both cases, $q_{\star}$ tends to $0$ as $\lambda$ tends to $+ \infty$. 
This is expected, as the obstacle acts more and more like a rigid wall. 
The root is unique when the initial data is supersonic.
\begin{figure}[h!t]
 \includegraphics[clip=true, trim=1cm 0cm 1.8cm 1.5cm, width=0.49\linewidth]{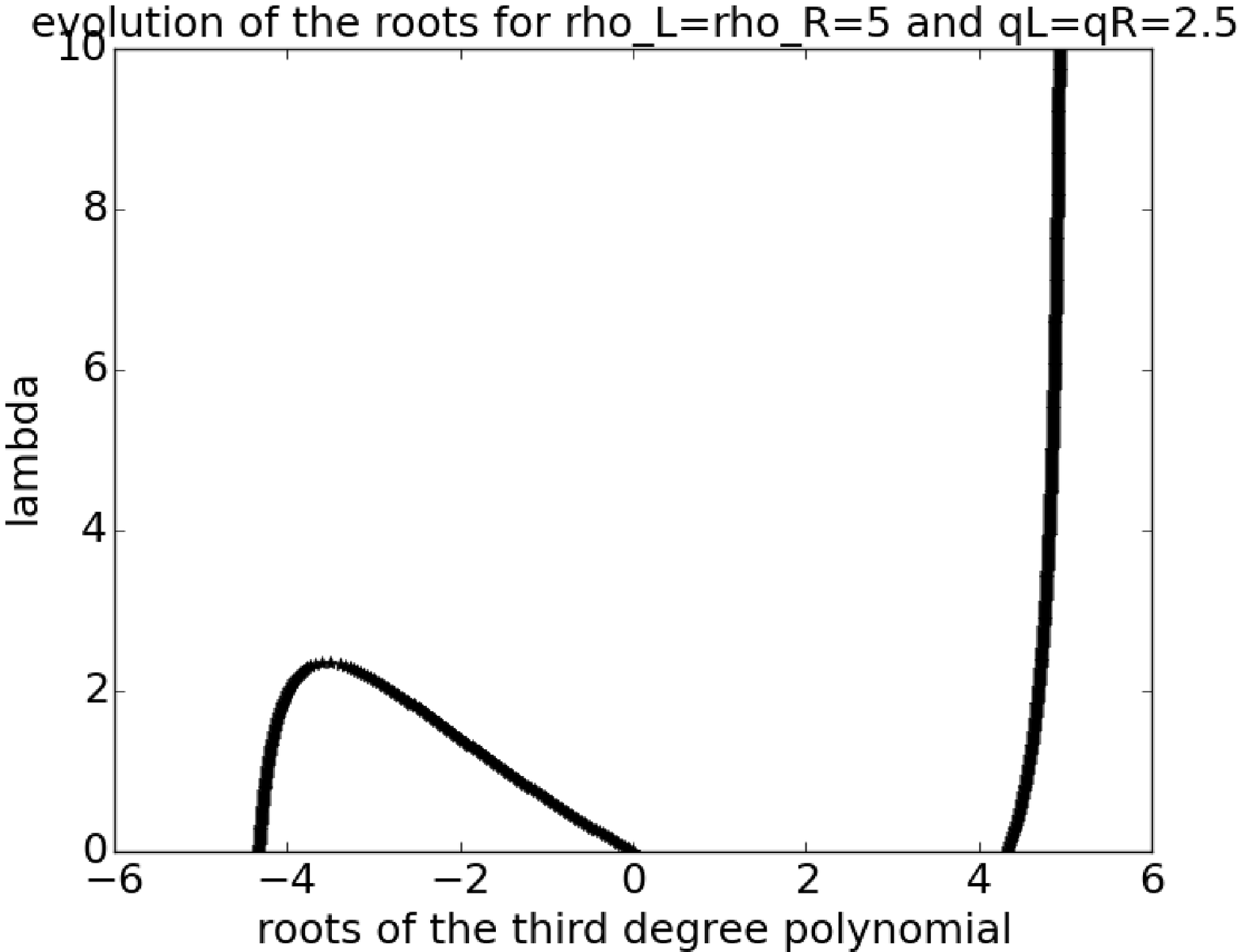}
 \includegraphics[clip=true, trim=1cm 0cm 1.8cm 1.5cm, width=0.49\linewidth]{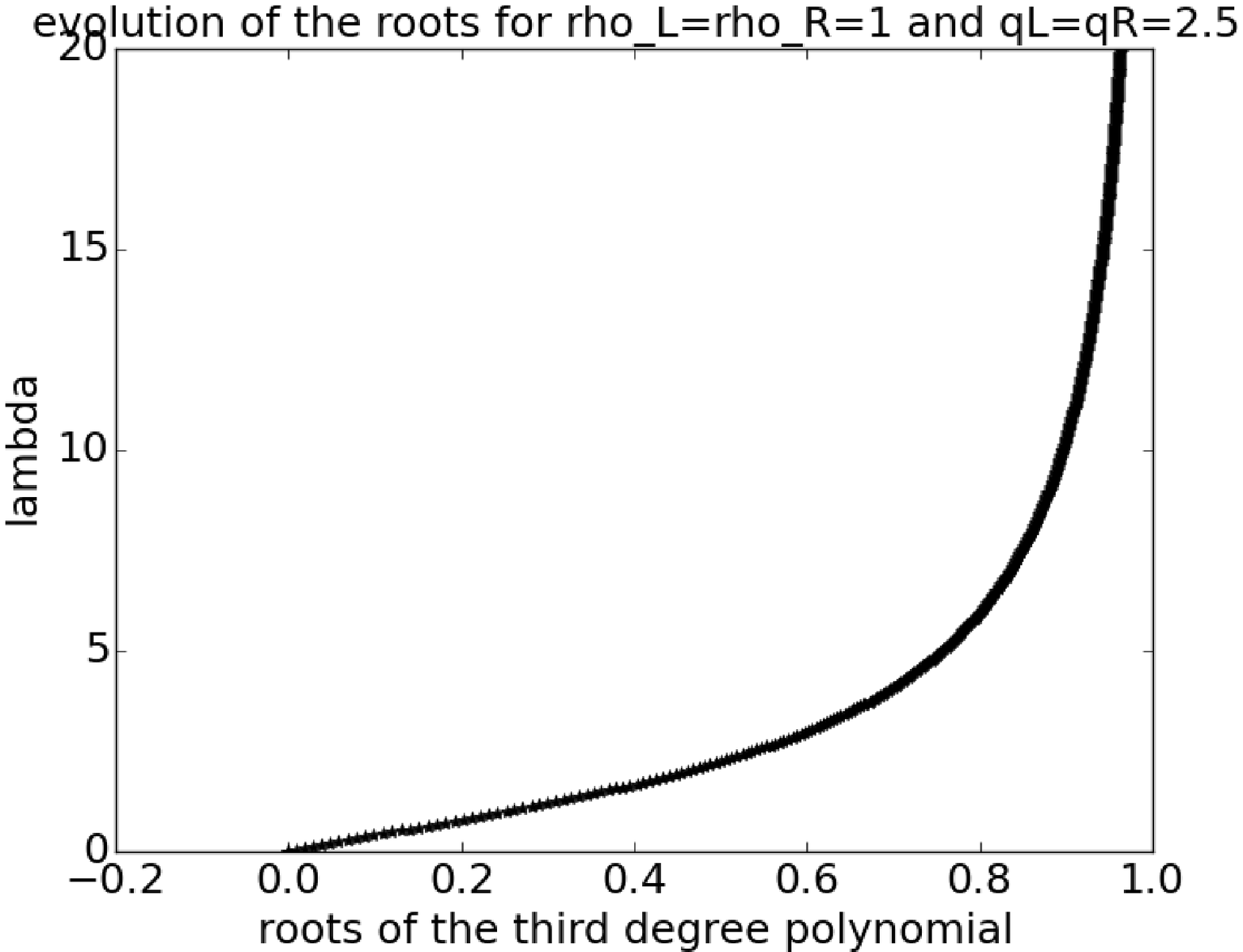}
 \caption{When the initial data is subsonic (left), there is $3$ root for small $\lambda$ and only $1$ for large $\lambda$. } \label{F: Roots}
\end{figure}

\subsection{Entropy fix} \label{S:fix}
It might happen that Equation~\eqref{eq: r part} does not admit any solution that verifies both the two last conditions of~\eqref{eq: Gpart} and the entropy condition~\eqref{eq: entrop 1/2}.
\begin{figure}[h!t]
	\centering
	\includegraphics[clip=true, trim=1cm 0cm 2cm 0cm, width=0.6\linewidth]{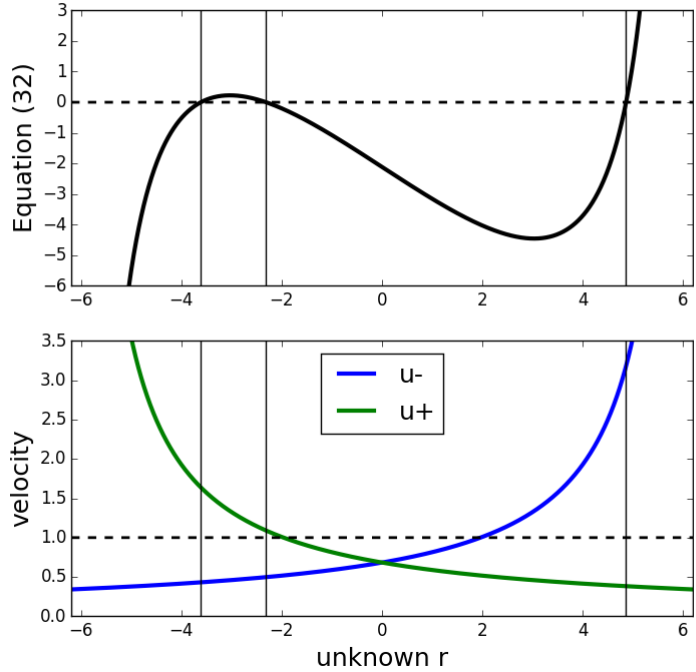}
	\caption{Top: the cubic polynomial associated with~\eqref{eq: r part}. Bottom: velocities of $u_{-}= \frac{q_{\star}}{\rho_{\star}-r}$ and $u_{+}= \frac{q_{\star}}{\rho_{\star}+r}$.} \label{F: OneSol}
\end{figure}
For example, if $\rho_{L}=4$, $\rho_{R}=10$, $q_{L}=1.9$ and $q_{R}=10$, $c=1$ and $\lambda=0.5$, Equation~\eqref{eq: r part} with $A=5$ admits three solutions. It can be seen on top of Figure~\ref{F: OneSol}. On the bottom of the figure, we see that the first and second solutions do not satisfy the third condition of~\eqref{eq: Gpart}. Indeed, $q_{\star}>0$, and the velocity $u_{-}$ at the entrance of the obstacle is subsonic, while the velocity $u_{+}$ at its exit is supersonic. It can be checked that the first and third ones do not verify~\eqref{eq: entrop 1/2}.

Whenever the numerical scheme picks a solution that violates an inequality condition of the interface conditions~\eqref{eq: Gpart}, we change $q_{\star}$ into $\frac{c}{\rho_{+}}$ (if $q_{\star}>0$) or $q_{\star}$ into $\frac{-c}{\rho_{-}}$ (if $q_{\star}<0$). 
This modification does not change $\rho_{\star}$, thus the conservation of mass still holds, but~\eqref{eq: r part} is relaxed. 
This fix is mandatory to approximate correctly sonic solution, i.e. solutions in which a wave interacts with the particle, see Figure~\ref{F: Resonant} below.

\section{Numerical tests}
In this section we investigate the accuracy of the proposed numerical method for different choices of interface conditions. The position of the interface is always $x=0$, but its relative position in the space interval is adapted to the solution of the actual problem. In all computations the Courant number is taken as $0.95$.
In all computations the speed of sound is $c=1$ and the considered intervals are discretized with $200$ cells, if not mentioned explicitly. 

We compare the results given by 
the Rusanov flux~\eqref{eq:fluxRus} and by the FORCE flux 
\begin{equation} \label{eq: fluxFORCE}
 g_{\text{FORCE}}(U_{L},U_{R})= \frac{g_{Rus}(U_{L},U_{R}) + f(U_{\star})}{2}, 
\end{equation}
where the middle state $U_{*}$ is given by~\eqref{eq:U*Rus}.

\subsection{Fluid/particle coupling}
The first series of tests we perform for the fluid/particle model \eqref{eq: Gpart}.
In the examples different Riemann problems at the interface are considered, 
which should cover all relevant scenarios. 

\subsubsection*{Subsonic (Test case $1$)} 
\begin{figure}[h!]
\centering
	\includegraphics[clip=true, trim=3cm 2cm 2cm 2cm, width=0.9\linewidth]{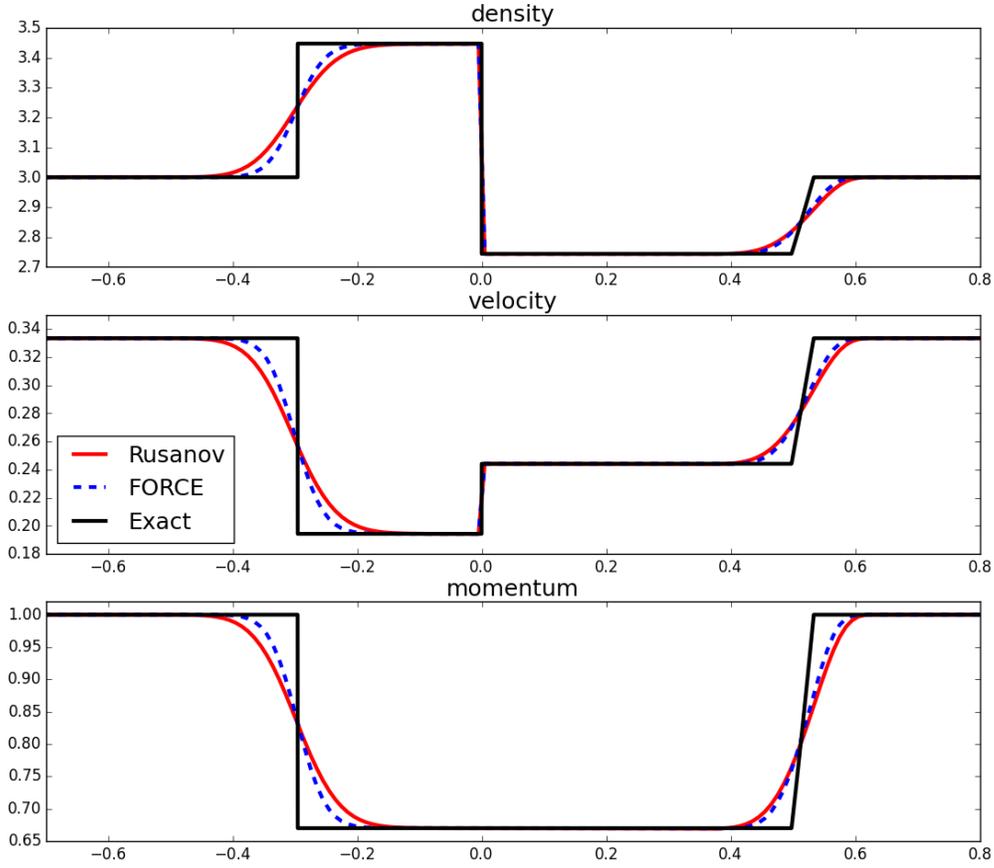}
	\caption{Test case $1$: subsonic initial data, subsonic solution. The obstacle partially blocks the flow.} \label{F: Sub}
\end{figure}
In the first test the initial data is chosen to be subsonic, i.e. $\rho_{L}=\rho_{R}=3$, $q_{L}=q_{R}=1$ and $\lambda=1$.
Thus, there is one wave moving to the left and one to the right.
 A large portion of the fluid is blocked in front of the obstacle and only a small percentage can pass. 
 As shown in Figure~\ref{F: Sub} this leads to a large density and small velocity in front of the obstacle, while small density and larger velocity behind it.
 We observe that the scheme approximates accurately the intermediate states at the interface.
 Only the shape of the waves is smeared out, as it is known from the Rusanov flux.
  \begin{figure}[h!t]
  \centering
  	\includegraphics[clip=true, trim=3cm 2cm 2cm 2cm, width=0.9\linewidth]{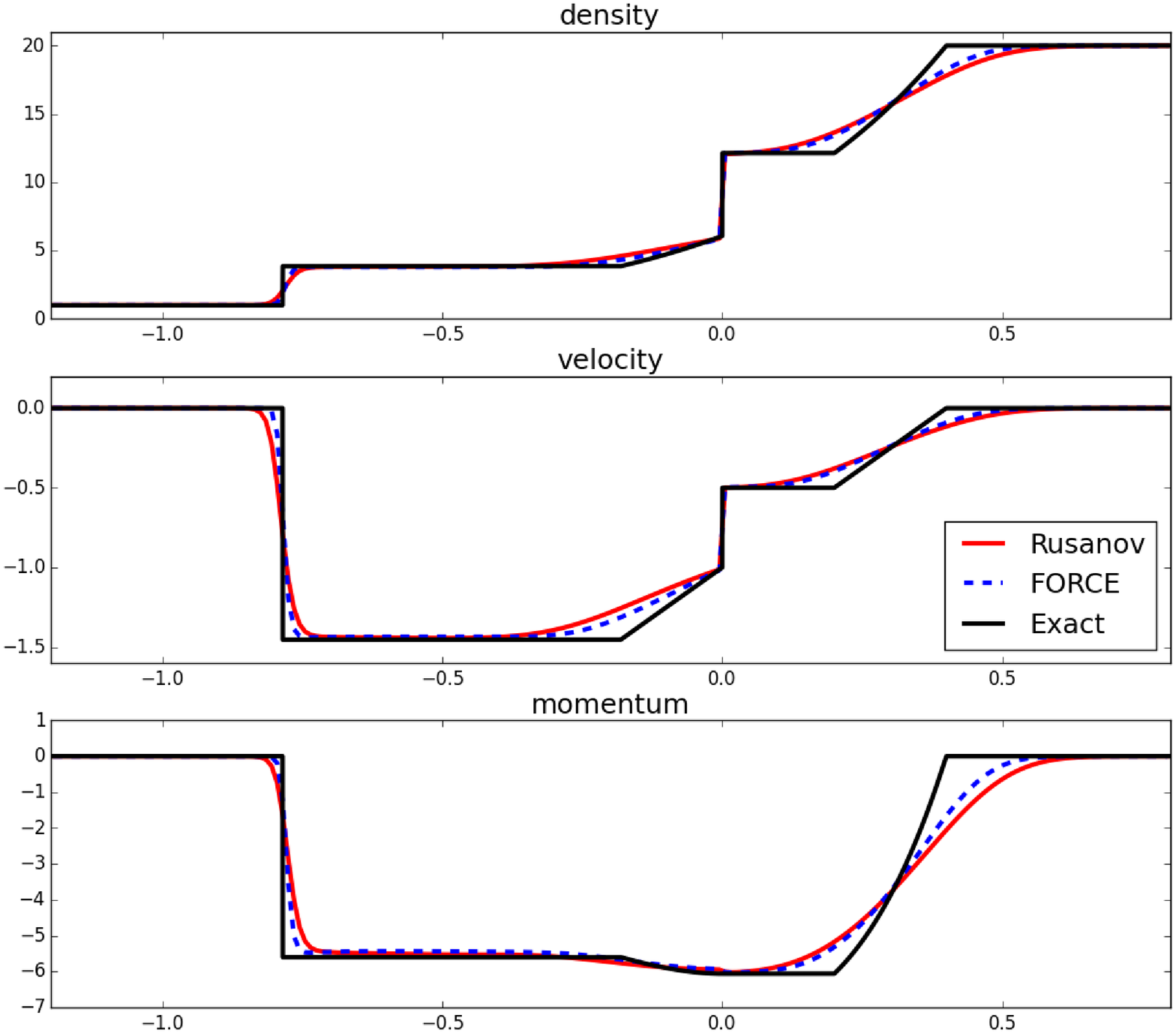}
  	\caption{Test case $2$: subsonic initial data, resonant solution. The solution contains three waves. One of them is a sonic rarefaction starting at the exit of the interface.} \label{F: Resonant}
  \end{figure}
  
  \subsubsection*{Subsonic to sonic (Test case $2$)}
  \begin{figure}[h!t]
  \centering
  	\includegraphics[clip=true, trim=3cm 2cm 2cm 2cm, width=0.9\linewidth]{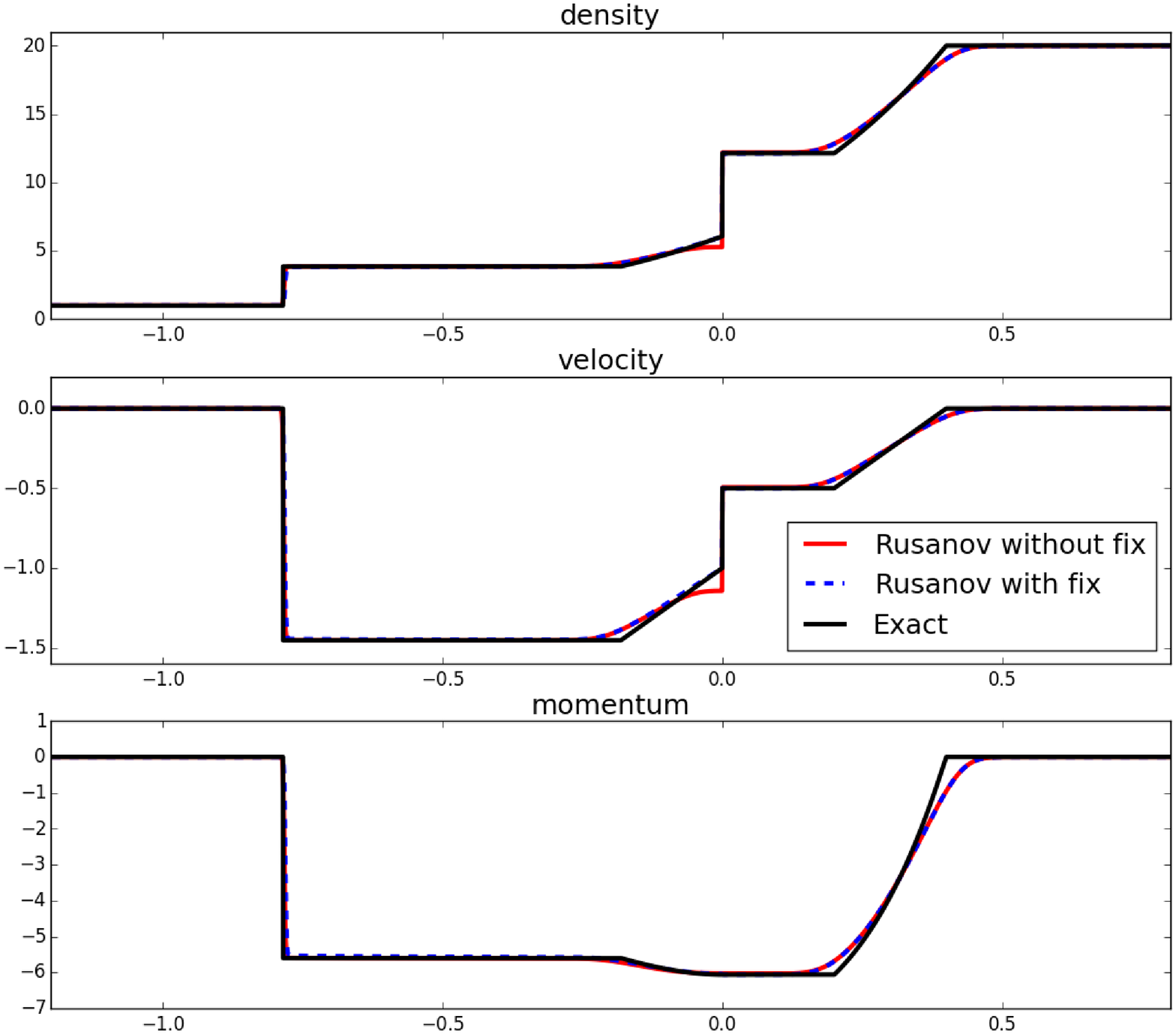}
  	\caption{Test case $2$: Convergence of the scheme with and without the fix of Section~\ref{S:fix} ($\Delta x= 0.001$)}
  	\label{F: Resonant2}
  \end{figure}
  At a coupling point it may occur that the initial data is subsonic, but its solution contains a sonic rarefaction wave. 
  Such a behaviour can be observed for the initial data is $\rho_{L}=1$, $\rho_{R}=20$, $q_{L}=q_{R}=0$ and $\lambda=0.5$, see Figure~\ref{F: Resonant}. 
This is due to the fact that the fluid / particle model is not strictly hyperbolic as shown in of~\cite{Agu14}.
It implies that the waves are not necessarily separated by a constant state. 
Solutions of this kind are difficult to approximate numerically.
In this test the fix of Section~\ref{S:fix} is active and without it the scheme converges towards an incorrect solution, see Figure~\ref{F: Resonant2} where $\Delta x= 0.001$.
\subsubsection*{Supersonic (Test case $3$)}
When the initial data is supersonic, the flow may remain supersonic if the drag coefficient $\lambda$ is small enough. For example if $\rho_{L}=\rho_{R}=1$, $q_{L}=q_{R}=3$ and $\lambda=1$, the obstacle does not slows down the flow enough to reach the sonic point. Thus, the perturbations due to the obstacle stay behind it, see Figure~\ref{F: Sup1}. 
\begin{figure}[h!t]
	\centering
	\includegraphics[clip=true, trim=3cm 2cm 2cm 2cm, width=0.9\linewidth]{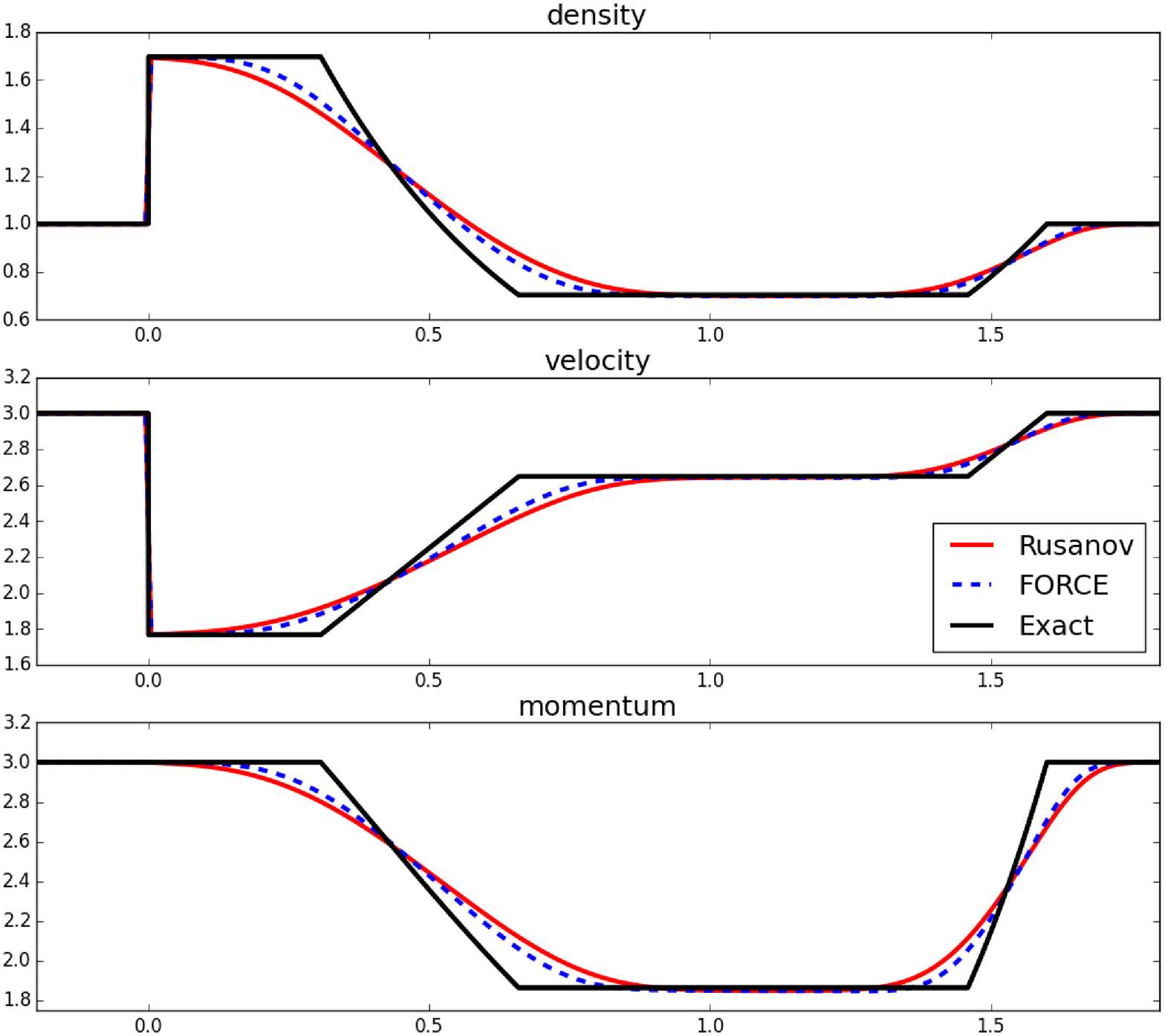}
	\caption{Test case $3$: supersonic initial data, supersonic solution. The flows is slowed downed by the obstacle, but remains supersonic.} \label{F: Sup1}
\end{figure}
\subsubsection*{Sonic (Test case $4$)}
 When $\lambda$ increases, we first obtain a sonic (or resonant) solution, the effect of the obstacle is strong enough to decreases the fluid's velocity below the speed of sound. 
 The fluid's velocity is subsonic in front of the obstacle and sonic behind it, see Figure~\ref{F: Sup2}. 
 The initial data is the same as in the previous case but with $\lambda=10$. 
 It is a very difficult test case, because the flow is resonant and the solution contains a slowly moving shock on the left side of the interface. 
 Slowly moving shocks are very difficult to capture numerically, see~\cite{AR97}. 
 A spurious peak appears in the momentum and pollutes the rest of the solution. 
 Note that this is due to the slowly moving shock and not to the method used at the interface; it explains the larger errors observed here. 
 In this particular test case the space interval is discretized with $800$ cells.
 \begin{figure}[h!t]
 \centering
 	\includegraphics[clip=true, trim=3cm 2cm 2cm 2cm, width=0.9\linewidth]{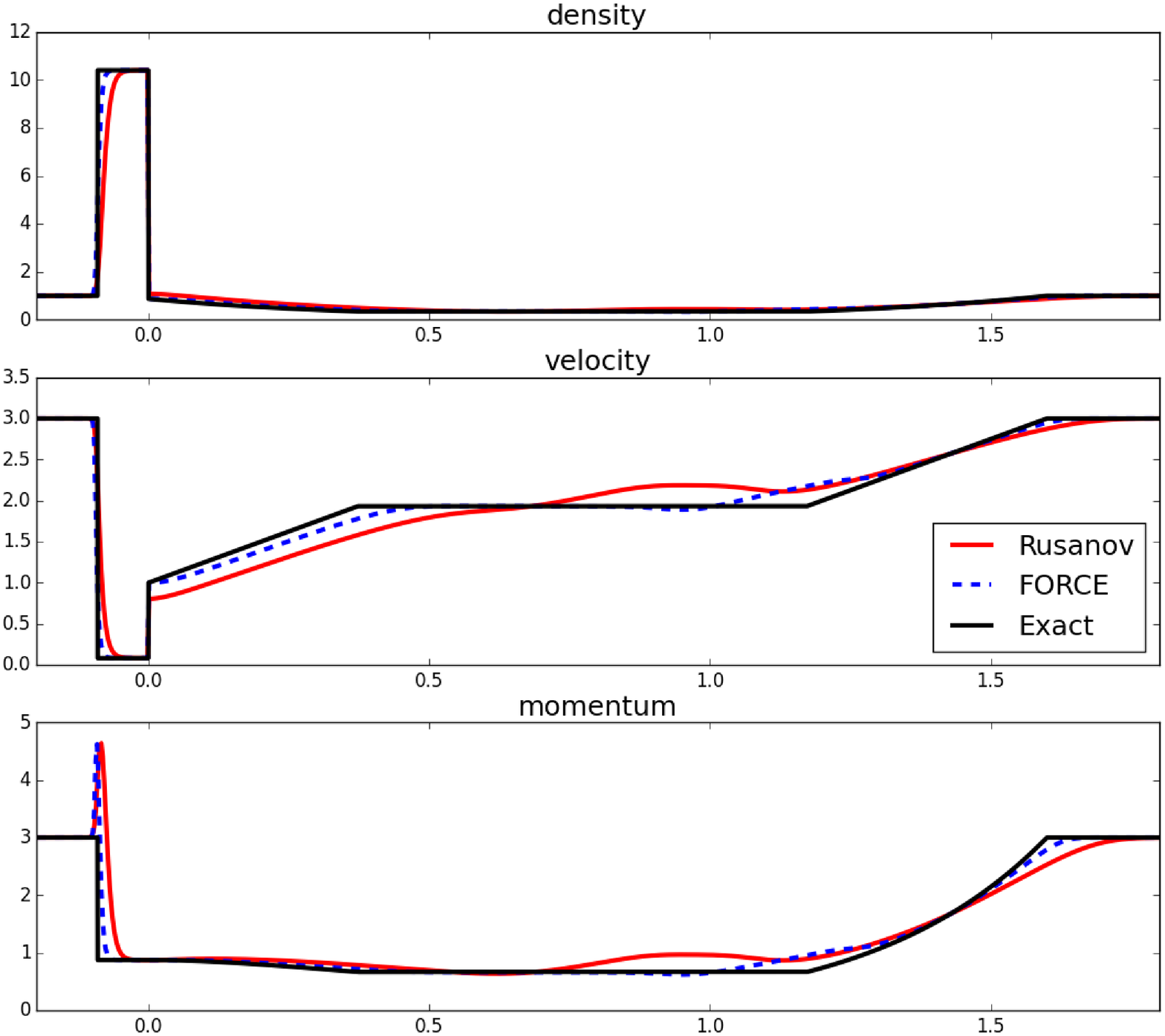}
 	\caption{Test case $4$: supersonic initial data, sonic solution. This case is somehow similar to test case $2$. $M=800$} \label{F: Sup2}
 \end{figure}
\subsubsection*{Sonic to subsonic (Test case $5$)}
 When $\lambda$ is very large, the flow might become subsonic on each sides of the obstacle, see Figure~\ref{F: Sup3}. 
 For this test case, $\lambda=10$, $\rho_{L}= \rho_{R}=2.5$, $q_{L}=q_{R}=3$. 
 The solution is approaches the case in which a rigid wall is placed at the interface.
 \begin{figure}[h!t]
 \centering
 	\includegraphics[clip=true, trim=3cm 2cm 2cm 2cm, width=0.9\linewidth]{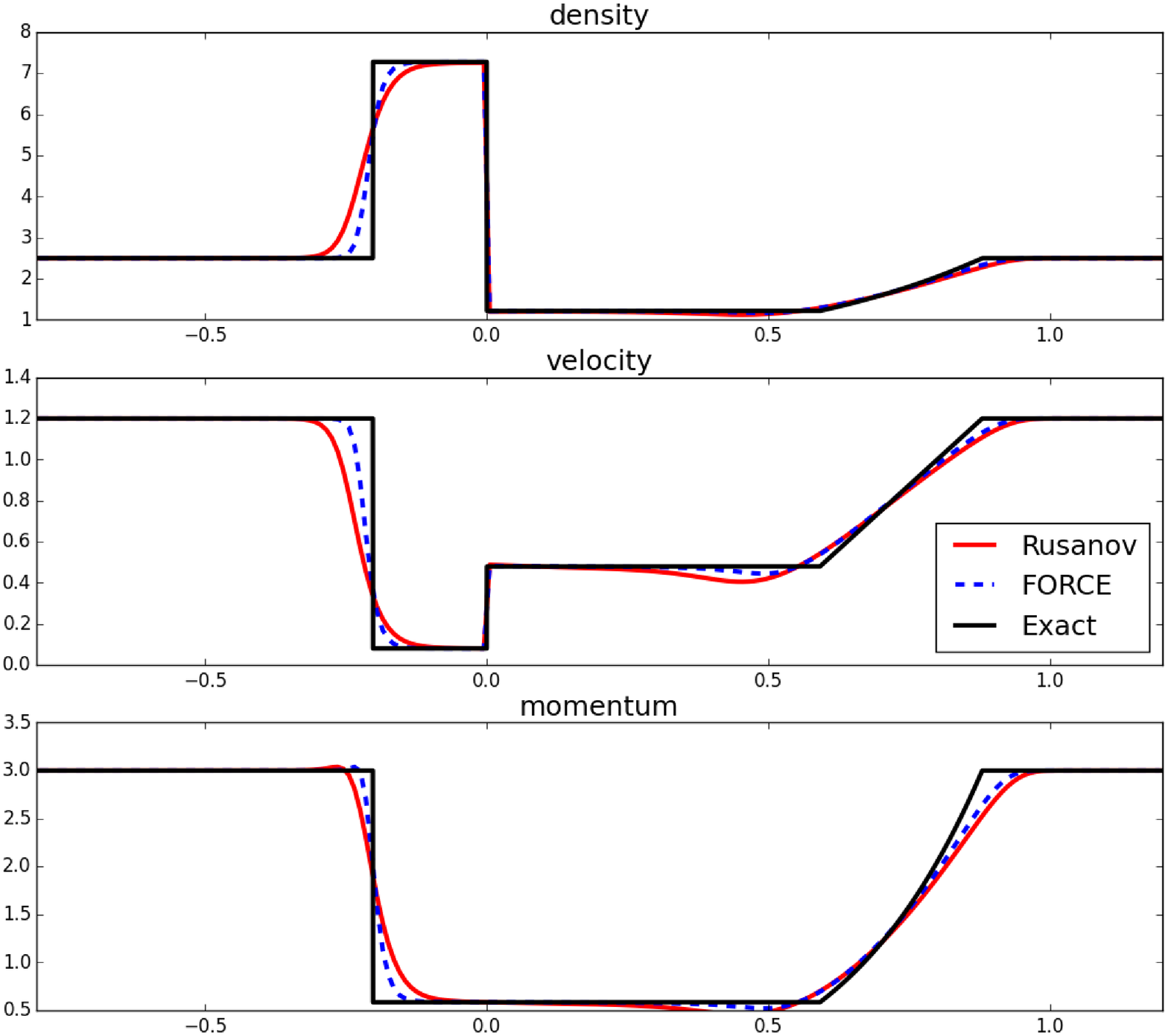}
 	\caption{Test case $5$: supersonic initial data, subsonic solution. The qualitative allure of the solution is similar to the one in test case $1$.} \label{F: Sup3}
 \end{figure}

\subsection{Fluid/particle coupling with heat exchange}
In this section we present a numerical simulation of model~\eqref{eq:heatex} obtained with Scheme~(\allowbreak\ref{eq:inG}-\ref{eq:update}-\ref{eq:no fluct}) with the Rusanov flux~\eqref{eq:fluxRus} and the FORCE flux~\eqref{eq: fluxFORCE}, when the parameters $\lambda$ and $\mu$ vary. The initial data is a constant subsonic flow: the fluid has initially a constant density of $4$, a constant velocity of $1$ and a constant pressure of $4$. The adiabatic exponent is $\gamma=1.5$ and we take $s_{P}=2$. The final time is $T=0.03$, $\rho_{0}$ is set to $1$, and the space interval $[-0.1, 0.1]$ is discretized with $500$ cells.
We considered the three following cases.

\subsubsection*{$\lambda=1$ and $\mu=0$ (Test case $6$): The obstacle partially blocks, but does not exchange heat with  the flow}
 The results are given on Figure~\ref{F:Heating1}. In that case, the qualitative behavior is the same than in the subsonic case presented on Figure~\ref{F: Sub}. Most of the fluid is stuck in front of the particle, where the fluid's velocity is small, and both the pressure and the internal energy  are large. A small part of the fluid manages to pass through the obstacle: the air after the obstacle has high velocity, and low pressure and internal energy.
 \begin{figure}[h!t]
 \centering
  \includegraphics[clip=true, trim=1.8cm 1cm 1.8cm 1cm, width=\linewidth]{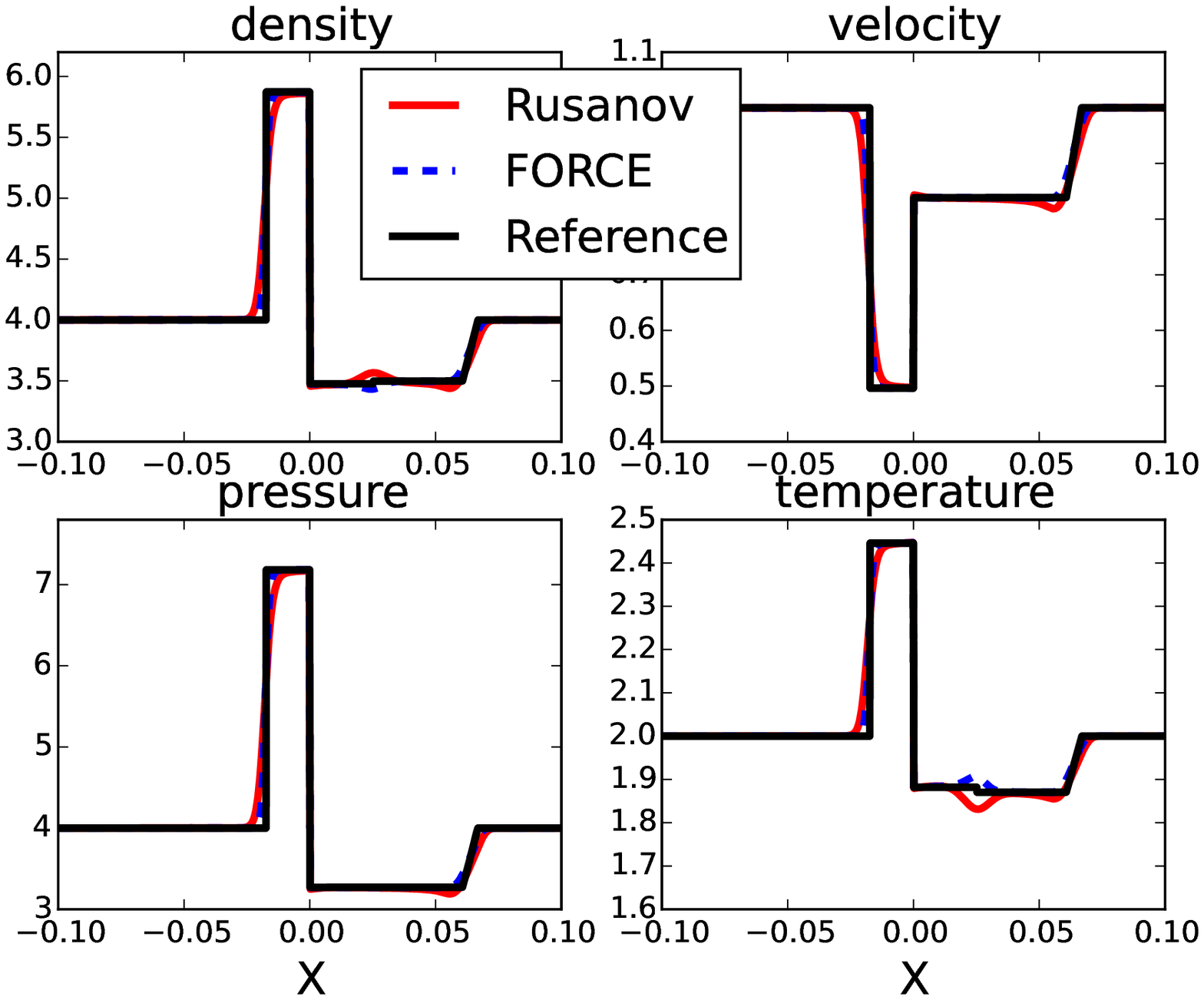}
 \caption{A subsonic constant flow is modified by an heat-insulating obstacle: $\lambda=1$, $\mu=0$} \label{F:Heating1}
\end{figure}

\subsubsection*{$\lambda=0$ and $\mu=0.5$ (Test case $7$): The obstacle does not block, but heats the flow} The results are given on Figure~\ref{F:Heating2}. 
We can think of this situation as an external heat source.
In that case, the main effect is that the temperature (i.e. the internal energy) of the fluid increases after the particle. 
This influences the other quantities according to the ideal gas law.
\begin{figure}[h!t]
\centering
 \includegraphics[clip=true, trim=1.8cm 1cm 1.8cm 1cm, width=\linewidth]{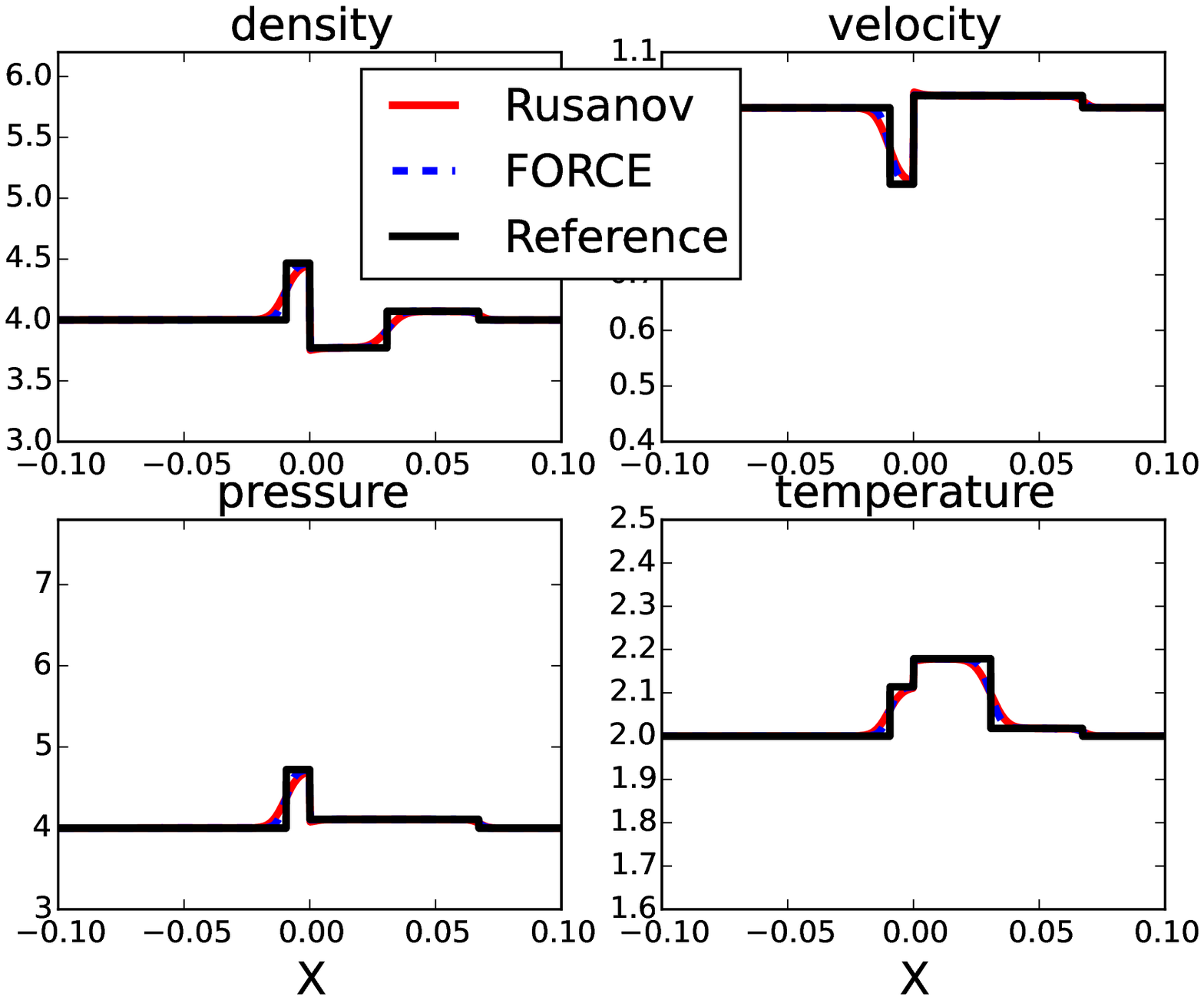}
 \caption{A subsonic constant flow is modified by a source that does not block the flow: $\lambda=0$, $\mu=0.5$, $s_{P}=2$} \label{F:Heating2}
\end{figure}

\subsubsection*{$\lambda=1$ and $\mu=0.5$ (Test case $8$): The obstacle both slows down and heats the flow} 
The results are given on Figure~\ref{F:Heating3}. The behavior is a mix between the two previous cases. 
In particular, depending on the ratio between $\lambda$ and $\mu$, the temperature after the obstacle can be larger or smaller than in front of it.
 \begin{figure}[h!t]
 \centering
 \includegraphics[clip=true, trim=1.8cm 1cm 1.8cm 1cm, width=\linewidth]{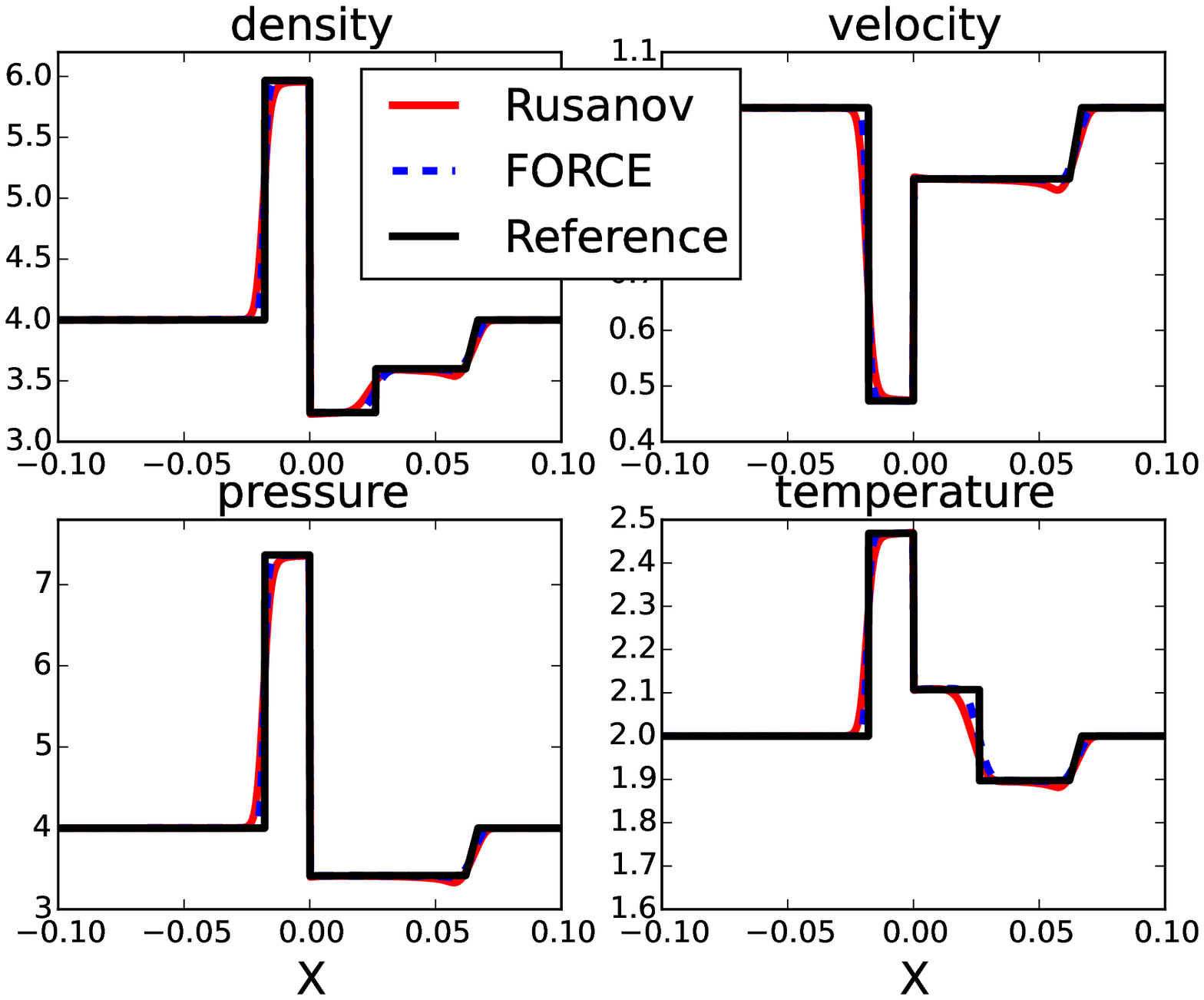}
 \caption{A subsonic constant flow is modified by a heating obstacle, that both slows down and exchange heat with the fluis: $\lambda=1$, $\mu=0.5$, $s_{P}=2$} \label{F:Heating3}
\end{figure}

Numerically, we observed that the scheme does not find a solution in the first iterations in time with our naive starting point. 
In that case it picks a solution as close as possible to $0$, thus it is somehow close to the scheme proposed in~\cite{Bor15}. 
After a few iterations, a solution is found at each time step.

The simulations presented here corresponds to a Riemann problem. 
At the present time, the solution of the Riemann problem for model~\eqref{eq:heatex} is not known. 
However, at the end of the simulation, we can use the left and right numeric traces and check that the structure of the solution described by (\ref{eq:Uminus}-\ref{eq:Uplus}-\ref{eq:inG}) is respected. These are used to construct the black ``Reference'' curves on Figures~\ref{F:Heating1}, \ref{F:Heating2} and~\ref{F:Heating3}.	

\subsection{Gas dynamics with different pressure laws}
In this section we present two numerical simulations of the coupling of two fluids with different pressure laws \eqref{eq:diffpressure}, each for the flux coupling~\eqref{eq:fluxcoupling} and the state coupling~\eqref{eq:statecoupling}. 
The initial data are the same than the ones proposed in~\cite{CGHMR15}: 
$$
\begin{array}{|c|cccccccc|}
\hline
  & \gamma_{L} & \rho_{L} & u_{L} & p_{L} & \gamma_{R} & \rho_{R} & u_{R} & p_{R}  \\
  \hline 
  \text{Test case 9} & 1.4 & 1.6 & 0.4 & 2.35 & 1.28 & 1.6 & 0.4 & 2.35 \\
  \hline 
  \text{Test case 10} & 1.4 & 1.6 & 0.4 & 2.35 & 1.28 & 1.4 & 0.4 & 1.9 \\
  \hline
\end{array}
$$
Scheme~\eqref{eq:update} is adapted to take into account the different physics, i.e. we use a flux $g_{L}$ with pressure law $p_{L}$ in the first two lines of~\eqref{eq:update}, and a flux $g_{R}$ with pressure law $p_{R}$ for the last two lines.

The results presented in Figure~\ref{F:diffpressure1} and~\ref{F:diffpressure2} are obtained with the FORCE flux (results for the Rusanov scheme are similar, but more diffusive). 
The first test case is an equilibrium for the  state coupling, which is exactly preserved by the numerical scheme. 
For the flux coupling, the scheme exactly preserves the conserved variables. 
For both coupling conditions, the results are comparable to those obtained in~\cite{CGHMR15} with a relaxation scheme.
 \begin{figure}[h!t]
 \includegraphics[clip=true, trim=2cm 1cm 2cm 1cm, width=0.95\linewidth,height=0.9\linewidth]{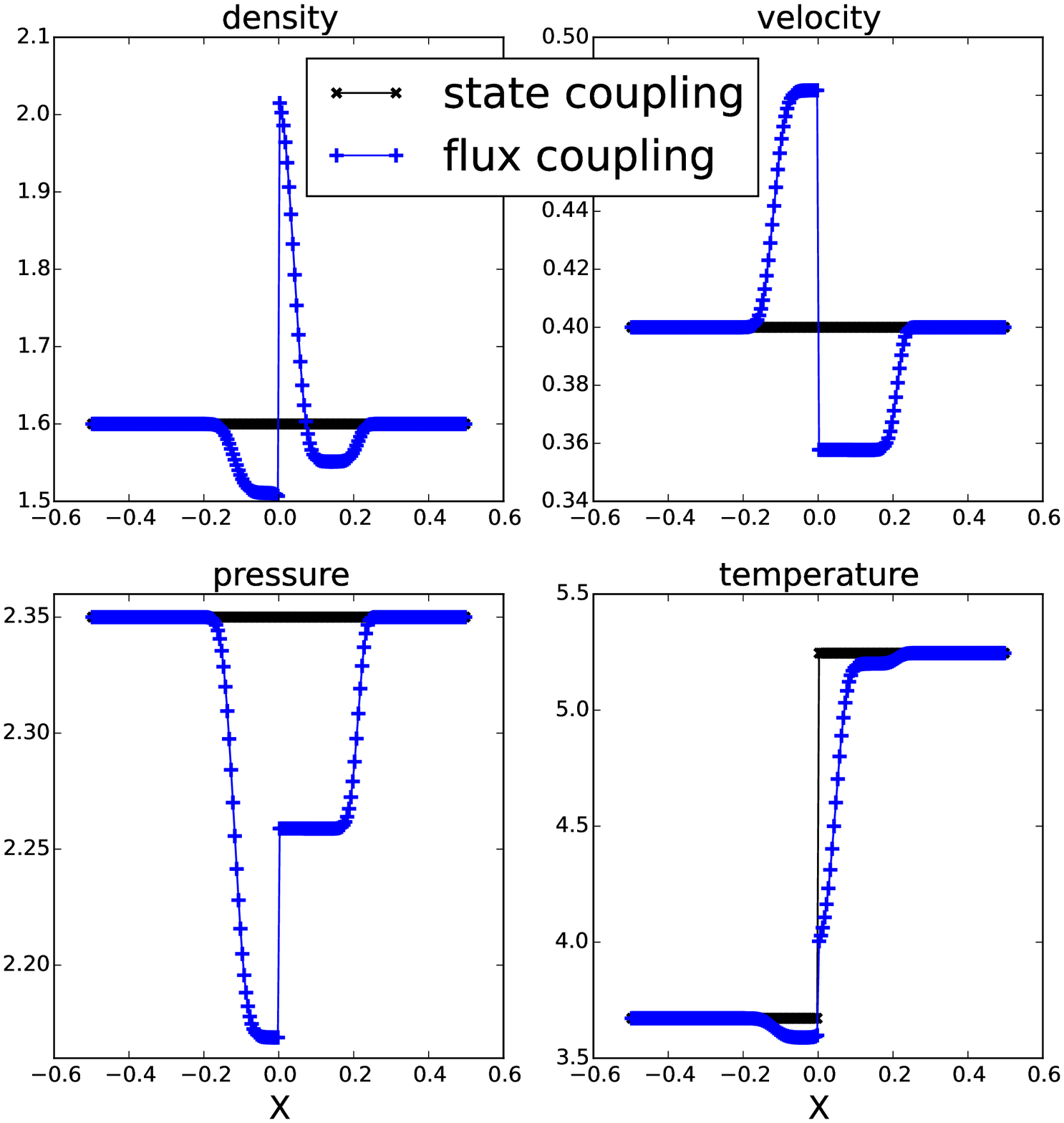}
 \caption{Test case $9$: flux and state couplings. The final time is $0.12$, the Courant number is set to $0.95$ and the space step is $\Delta x= 0.005$} \label{F:diffpressure1}
\end{figure}
 \begin{figure}[h!t]
 \includegraphics[clip=true, trim=2cm 1cm 2cm 1cm, width=0.95\linewidth,height=0.9\linewidth]{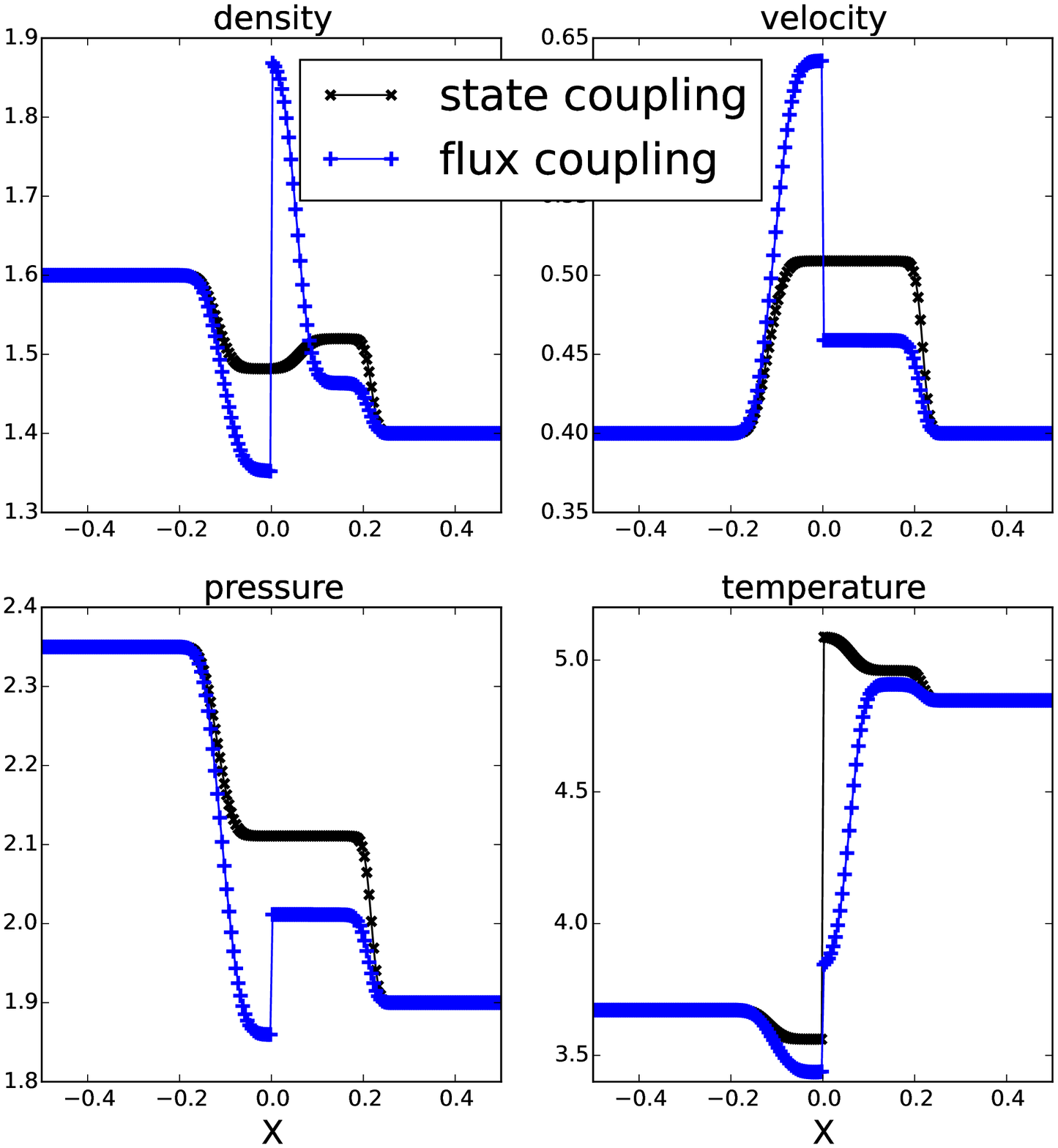}
 \caption{Test case $10$: flux and state couplings. The final time is $0.12$, the Courant number is set to $0.95$ and the space step is $\Delta x= 0.005$} \label{F:diffpressure2}
\end{figure}

\subsection{Barotropic flows in a nozzle with piecewise constant cross-section} 
As last example we consider model~\eqref{eq:nozzle} with the interface conditions~\eqref{eq:Gnozzle} and the pressure law $p(\tau)=\tau^{-3}$. 
We reproduce the two test cases proposed in~\cite{CKS14}. The initial conditions  are given by
$$
\begin{array}{|c|ccc|}
\hline
  & \alpha_{L} & \rho_{L} & w_{L}   \\
  \hline 
  \text{Test case 11} & 0.3 & 0.206052848877390 & -0.003218270138816 \\
  \hline 
  \text{Test case 12} & 1 & 0.988056834959612 & 0.125759712385390  \\
  \hline
   & \alpha_{R} & \rho_{R} & w_{R}  \\
   \hline
   \text{Test case 11} & 0.4 & 0.099 & -0.015876669673295 \\
  \hline 
  \text{Test case 12} & 100 & 1.01 & 0.018403108075689  \\
  \hline
\end{array}
$$
The final time is $1$ for test case $11$ and $0.15$ for test case $12$.
The space interval is $[-0.5, 0.5]$. 
The results for a discretization with $\Delta x= 0.01$ are given on Figures~\ref{F:nozzle1} and~\ref{F:nozzle2}. 
as the exact traces around the interface are known the numerical errors for test case $11$ are given in the table below.
$$ 
\begin{array}{|c|cccc|}
\hline
 	& \rho_{-} & w_{-} & \rho_{+} & w_{+} \\
\hline
\text{exact} & 0.1440929013128 & 0.10409950707725 & 0.15 & 0.075 \\
\hline
\begin{array}{c}
 \text{Rusanov} \\
 \Delta x=10^{-2}
\end{array}
 & 6.22\times 10^{-3} & 1.36 \times 10^{-4} & 3.09 \times 10^{-4} & 6.931 \times 10^{-5} \\
 \hline
 \begin{array}{c}
 \text{Rusanov} \\
  \Delta x=10^{-3}
\end{array}
 & 8.83 \times 10^{-6} & 8.26 \times 10^{-5} & 1.86 \times 10^{-5} & 5.48 \times 10^{-5} \\
 \hline
\begin{array}{c}
 \text{FORCE} \\
  \Delta x=10^{-2}
\end{array}
 & 1.49 \times 10^{-4} & 1.38 \times 10^{-4} & 1.54 \times 10^{-4} & 1.15 \times 10^{-4} \\
 \hline
 \begin{array}{c}
 \text{FORCE} \\
  \Delta x=10^{-3}
\end{array}
 & 4.54 \times 10^{-6} & 4.59 \times 10^{-5} & 9.99 \times 10^{-6} & 3.04 \times 10^{-5} \\
 \hline
\end{array}
$$
Both schemes produce good approximations, which converge to the exact solution when the grid is refined.

The next table regroups the exact traces and the errors around the interface for test case $12$. 
It clearly shows that the scheme is able to capture correctly the solution when the jump in the cross section is large. This is even true for the dissipative Rusanov scheme with a coarse discretization. 
$$ 
\begin{array}{|c|cccc|}
\hline
 	& \rho_{-} & w_{-} & \rho_{+} & w_{+} \\
\hline
\text{exact} & 0.9980372070299 & 0.108472909864928 & 1 & 0.0010826 \\
\hline
\begin{array}{c}
 \text{Rusanov} \\
  \Delta x=10^{-2}
\end{array}
 & 4.96\times 10^{-7} & 2.27 \times 10^{-5} & 2.45 \times 10^{-6} & 2.91 \times 10^{-7} \\
 \hline
 \begin{array}{c}
 \text{Rusanov} \\
  \Delta x=10^{-3}
\end{array}
 & 6.35 \times 10^{-7} & 6.3 \times 10^{-7} & 6.63 \times 10^{-7} & 7.57 \times 10^{-9} \\
 \hline
\begin{array}{c}
 \text{FORCE} \\
  \Delta x=10^{-2}
\end{array}
 & 1.68 \times 10^{-6} & 1.63 \times 10^{-7} & 1.69 \times 10^{-6} & 3.35 \times 10^{-8} \\ 
 \hline
 \begin{array}{c}
 \text{FORCE} \\
  \Delta x=10^{-3}
\end{array}
 & 4.82 \times 10^{-7} & 3.13 \times 10^{-8} & 4.83 \times 10^{-7} & 1.22 \times 10^{-9} \\ 
\hline
\end{array}
$$

 \begin{figure}[h!t]
 \includegraphics[clip=true, trim=1.5cm 0.5cm 2cm 0cm, width=0.95\linewidth,height=0.42\linewidth]{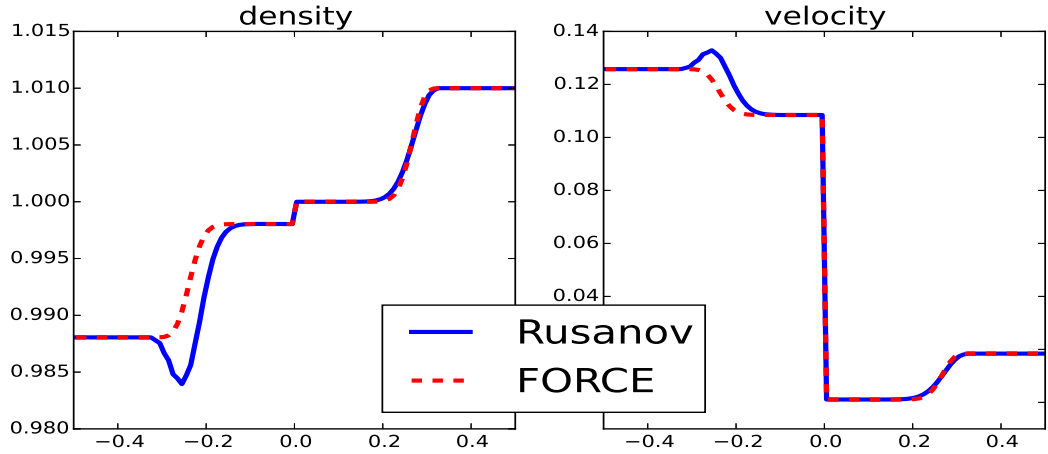}
 \caption{Test case $7$: a subsonic Riemann problem for model~\eqref{eq:nozzle}} \label{F:nozzle1}
\end{figure}
\begin{figure}[h!t]
 \includegraphics[clip=true, trim=1.5cm 0.5cm 2cm 0cm, width=0.95\linewidth,height=0.42\linewidth]{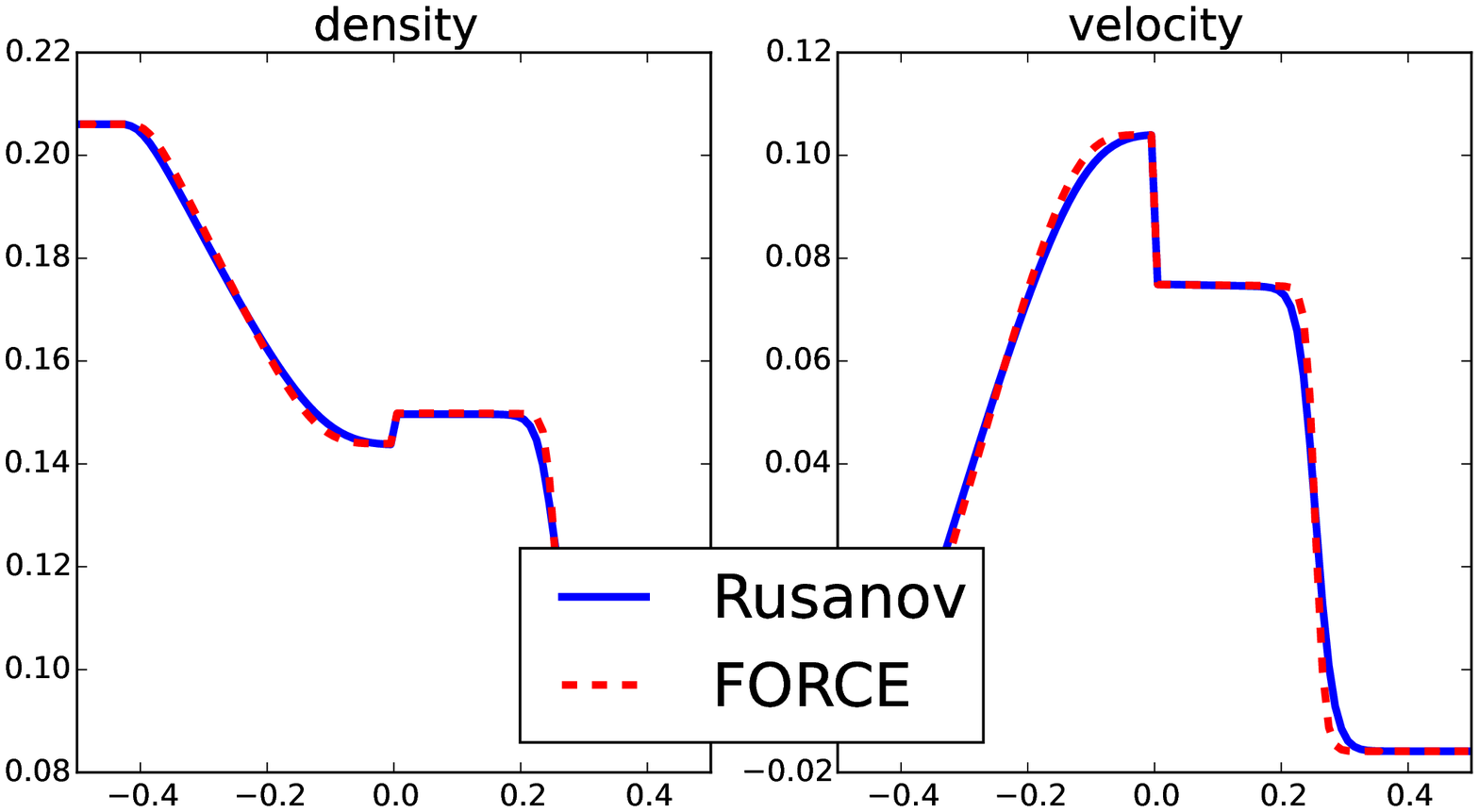}
 \caption{Test case $8$: a subsonic Riemann problem for model~\eqref{eq:nozzle} with a large jump in the cross section} \label{F:nozzle2}
\end{figure}

\section*{Conclusion}
We presented a numerical approach which is able to deal with interface conditions for $1$-dimensional hyperbolic systems of conservation laws. 
The scheme is derived in such a way that the structure of the Riemann problem is mimicked at the numerical level ``as good as possible''. 
The interface conditions are exactly taken into account. 
In the general case, it is not possible to enforce that no waves enter the junction. 
Thus we relaxed the classical coupling procedure by only requiring that the effects of all entering waves cancel each other.

The approach is analyzed in detail for the Godunov and the Lax-Friedrichs scheme and it is proved that the resulting approximate Riemann solver verifies an entropy inequality.
This scheme was tested on four different models.
In the first one, by introducing a small modification, also sonic flows can be approximated.
The other test cases include a model with heat exchange and a coupling of two different conservation laws on each side of the interface. 
This solver is easy to implement and the applications illustrate its flexibility.

\textbf{Acknowledgment:} The first author was partially supported by the University Paris Sud and by the Labex Archimède of Aix-Marseille Université.

\bibliographystyle{plain}
\bibliography{biblioInterfaceWithoutRP}

\end{document}